\documentclass[10pt]{amsart}
\usepackage{latexsym,amsmath,amssymb}
\usepackage{mathrsfs}
\usepackage{mathabx}
 \usepackage{indentfirst}
 \renewcommand{\epsilon}{\varepsilon}
 \newcommand{\newsection}[1]
  {\section{#1}\setcounter{theorem}{0} \setcounter{equation}{0}\par\noindent}

   \newtheorem{theorem}{Theorem}[section]
   \newtheorem{lemma}[theorem]{Lemma}
 \newtheorem{corr}[theorem]{Corollary}
 
 \newtheorem{proposition}[theorem]{Proposition}
 \newtheorem{deff}[theorem]{Definition}
 \newtheorem{remark}[theorem]{Remark}

  \numberwithin{equation}{section}

 \newcommand{\bth}{\begin{theorem}}
 \newcommand{\ble}{\begin{lemma}}
 \newcommand{\bcor}{\begin{corr}}
 \newcommand{\bdeff}{\begin{deff}}
 \newcommand{\bprop}{\begin{proposition}}
 \def\be{\begin{equation}}
\def\ee{\end{equation}}
\def\bt{\begin{theorem}}
\def\et{\end{theorem}}
\def\ba{\begin{array}}
\def\ea{\end{array}}
\def\bl{\begin{lemma}}
\def\el{\end{lemma}}

 \newcommand{\ele}{\end{lemma}}
 \newcommand{\ecor}{\end{corr}}
 \newcommand{\edeff}{\end{deff}}
 
 \newcommand{\eprop}{\end{proposition}}

 \renewcommand{\Pi}{\varPi}

 \renewcommand{\epsilon}{\varepsilon}

\newcommand\tbbint{{-\mkern -16mu\int}}

\newcommand\dbbint{{-\mkern -19mu\int}}

\newcommand\bbint{
{\mathchoice{\dbbint}{\tbbint}{\tbbint}{\tbbint}}
}

\title[Global Solution for  the incompressible  Navier-Stokes equations] { Global  Solution for the incompressible  Navier-Stokes equations with a class of large data in $BMO^{-1}(\mathbb{R}^3)$  }

\date{\today}

\begin{document}
\maketitle

\centerline{
\author{
Du Yi
  \footnote{Department of Mathematical Sciences, Jinan University, Guangzhou 510632, China.
 {\it Email: duyidy@gmail.com; duyidy@jnu.edu.cn}}
 }
 \and
  Zhou Yi
  \footnote{ *Corresponding Author: School of Mathematical Sciences, Fudan University, Shanghai, China. and Department of Mathematical Sciences, Jinan University, Guangzhou 510632, China.
 {\it Email: yizhou@fudan.edu.cn}
 }
  }

\begin{abstract}In this paper, we shall establish  the global well-posedness, the space-time analyticity of the Navier-Stokes equations
 for a class of  large  periodic data  $u_0 \in BMO^{-1}(\mathbb{R}^3)$.
  This improves the classical result of  Koch \& Tataru  \cite{koch-tataru},
  for the  global well-posedness  with small initial data  $u_0 \in BMO^{-1}(\mathbb{R}^n)$.
%Subsequently,  the spatial and time  analyticity of the Koch \& Tataru solution  have been presented by  Germain-Pavlovic'-Stffilani \cite{germain} and the first author \cite{du1} when the initial data $u_0 \in BMO^{-1}(\mathbb{R}^n)$ small enough.

%
%and $ \|u_0-\lambda \Delta^{-1}\nabla\times u_0\|_{BMO (\mathbb{R}^n)}\leq \epsilon_0 <<1$ where $\lambda$ and $\epsilon_0$ are fixed  constants, $M_0$ is a uniform constant independent of $\epsilon_0$.
\end{abstract}

\newsection{Introduction}
This paper is devoted to study  the following incompressible Navier-Stokes equations with the periodic data  on the domain  $(x,t) \in \mathbb{R}^3\times R^+$.
\begin{equation}\label{1.1}
\begin{cases}
U_t+U\cdot\nabla U- \Delta U+\nabla P=0, \qquad \mbox{in} \ \mathbb{R}^3\times (0,+\infty),
\\
\nabla \cdot U=0, \\
 U(x,t)|_{t=0}= u_0(x).
\end{cases}
\end{equation}

At the very beginning, we shall recall some correlated research history about the incompressible Navier-Stokes equations.
In \cite{leray2},  Leray  proved the local well-posedness for  strong solutions and  for any finite square-integrable initial data  there exists  a (possibly not unique) global in time weak
solution in $\mathbb{R}^n$.  Moreover, for the  case of two space dimensions, he proved in \cite{leray1} the uniqueness of the weak solution.
Subsequently, in the work of Fujita-Kato\cite{fujita-kato}, they proved the local well-posedness for strong solutions to  the Navier-Stokes equations in a scaling invariant  space  $\dot{H}^{\frac{n}2-1}$. The scaling-invariance  in the context of the Navier-Stokes equations is  as follows.
Define
\begin{equation}\label{1.2}
(U_\lambda, P_\lambda)(x,t)= (\lambda U(\lambda x,\lambda^2t),\lambda^2P(\lambda x,\lambda^2t)),
\end{equation}
if the pair $(U(x,t), P(x,t))$  solves the incompressible
 Navier-Stokes equations, then $(U_\lambda(x,t), $ $    P_\lambda(x,t))$ is also a pair of solution to  the incompressible Navier-Stokes equations with initial data $U_\lambda(x,0)=\lambda u_0(\lambda x)$.
The spaces which are invariant under such a scaling are also called critical spaces.
The study of the Navier-Stokes equations in critical spaces was initiated by Kato\cite{kato} and
continued by many authors, see \cite{Cannone, giga,planchon} etc. In 2001 Koch-Tataru \cite{koch-tataru} proved the
existence of solutions to Navier-Stokes equations  in $\mathbb{R}^n$ when the initial data is in $BMO^{-1}$,  see also \cite{Lemarie-Rieusset}.  Subsequently, the space and time analyticity of the Koch-Tataru solution  have been studied by \cite{germain,du1} and \cite{miura-sawada}.    The space $BMO^{-1}$  has a special role since it  is the largest
critical space for Navier-Stokes equations with  well-posedness  results available \cite{bourgain}.
 Hereafter, we call the solution presented by Koch \& Tataru \cite{koch-tataru} as Koch-Tataru solution.

The results listed above are mostly concerned the  system \eqref{1.1}  with the  initial data small enough. To prove whether global existence or finite time blow up for a   large data is a famous open problem. Recently, in Lei-Lin-Zhou \cite{lei-lin-zhou} the authors have
 proved  the global well-posedness  with a class of  large data in the energy space which includes the Beltrami flow.
This paper is to study the global well-posedness for the 3D incompressible Navier-Stokes equations with general  initial data in the largest critical space $BMO^{-1}$.

%In this paper, we study the system on a periodic domain reads as
%We study the 3D incompressible Navier-Stokes equations in $ [-\pi,\pi]^3\times \mathbb{R}_+=T^3\times \mathbb{R}_+$

Before stating the main results, we shall  give  some definitions and notations.  Let
\begin{equation}\label{Q}
Q(y_0,r)=B(y_0,r)\times (0,r^2]
\end{equation}
be the space-time ball. For $(x,t)\in Q(y_0,r)$ means $x\in B(y_0,r)$ and $0< t\leq  r^2$,
where $B(y_0,r)\subset{\mathbb{R}^n}$ is a n-dimensional space ball  centered at $y_0\in \mathbb{R}^n$ and radius $r$.
\begin{deff}Let f be a tempered distribution,
 $W$ be the solution of $W_t-\Delta W=0$ with initial data $f$. Denote
\begin{equation}\nonumber
[f]_{BMO(\mathbb{R}^n)}=\sup\limits_{y_0\in \mathbb{R}^n  }\bigg(r^{-n}\int_{Q(y_0,r)} |\nabla W|^2 dtdy \bigg)^{\frac12},
\end{equation}
we say the function $f\in L_{loc}^1(\mathbb{R}^n)$ is in $BMO$ if the semi-norm $[f]_{BMO}$ is finite.

If there exist $g_i\in  BMO$ and $f=\sum_{i=1}^n \frac{\partial g_i}{\partial x_i}$, denote
\begin{equation}\nonumber
%\|f\|_{BMO^{-1}}=\sup\limits_{y_0\in \mathbb{T}^3,B(x,r)\subset\mathbb{T}^3 }\bigg(r^{-n}\int_{Q(y_0,r)}|W|^2dtdy\bigg)^\frac12.
[f]_{BMO^{-1}(\mathbb{R}^n)}=\inf  \sum_{i=1}^n\|g_i\|_{BMO(\mathbb{R}^n)},
\end{equation}
we say  $f\in  BMO^{-1}$ if   the above norm $[f]_{BMO^{-1}}$ is finite. It is easy to see that an equivalent definition for $BMO^{-1}$ is
\begin{equation}\nonumber
[f]_{BMO^{-1}(\mathbb{R}^n)}=\sup\limits_{y_0\in \mathbb{R}^n  }\bigg(r^{-n}\int_{Q(y_0,r)} |W|^2 dtdy \bigg)^{\frac12}.
\end{equation}
Clearly the divergence of a vector field with components in $BMO$ is in
$BMO^{-1}$. See more details in Koch-Tataru \cite{koch-tataru} and Chap11 in \cite{Lemarie-Rieusset}.

Similarly, we define the space $BMO^{-2} $.
If there exist $\bar{g}_i\in  BMO^{-1}$ and $f=\sum_{i=1}^n \frac{\partial \bar{g}_i}{\partial x_i}$, denote
\begin{equation}\nonumber
%\|f\|_{BMO^{-1}}=\sup\limits_{y_0\in \mathbb{T}^3,B(x,r)\subset\mathbb{T}^3 }\bigg(r^{-n}\int_{Q(y_0,r)}|W|^2dtdy\bigg)^\frac12.
[f]_{BMO^{-2}(\mathbb{R}^n)}=\inf  \sum_{i=1}^n\|\bar{g}_i\|_{BMO^{-1}(\mathbb{R}^n)},
\end{equation}
we say  $f\in  BMO^{-2}$ if   the above norm $[f]_{BMO^{-2}}$ is finite.

 \end{deff}

%\begin{deff}
%Similarly, if  there exists $h=(h_1,\cdots,h_n)\in BMO(\mathbb{T}^3)$ with  $f_i=\sum_{j,k=1}^n\frac{\partial^2 h_i}{\partial_j\partial_k}$, denote
%\begin{equation}\nonumber
%[f]_{BMO^{-2}(\mathbb{T}^3)}=\inf  \sum_{i=1}^3\|h_i\|_{BMO},
%\end{equation}
%we say the function $f$ is in $BMO^{-2}(\mathbb{T}^3)$ if the semi-norm $[f]_{BMO}$ is finite.
% \end{deff}

In Koch\& Tataru \cite{koch-tataru}, the following results have been presented:
\begin{theorem}\label{koch-tataru}
(Koch-Tataru\cite{koch-tataru}). If  $\|u_0\|_{BMO^{-1}}$ is small enough,  then the  equation  \eqref{1.1}  admits a unique pair of global solution.  \end{theorem}

Our main result states as following:
\begin{theorem}\label{main theorem1}
Let  $\lambda$ , $0<b<1$ be  two arbitrary given constants and  $M_0$ be an arbitrary positive constant,  suppose that the initial data $u_0$ of system \eqref{1.1} is a periodic function with
\begin{equation}\label{1.4}
\int_{\mathbb{T}^3} u_0(y)dy=0,
\end{equation}
and
\begin{equation}\label{1.5}
\|u_0\|_{BMO^{-1}(\mathbb{R}^3)}\leq M_0,
\end{equation}
then there exists a small constant $\varepsilon=\varepsilon (b, M_0)$ which depends on $M_0$ and $b$, but independent of  $\lambda$ , such that the system \eqref{1.1} admits a unique pair of  global periodic solutions  provided
\begin{equation}\label{1.6}
\| \nabla\times u_0-\lambda u_0\|_{BMO^{-2}(\mathbb{R}^3)}\leq \varepsilon {\left \langle \lambda \right \rangle}^{-b},
\end{equation}
where $\left \langle \lambda \right \rangle=\sqrt{1+{\lambda}^2}$.

Here and hereafter,  $\mathbb{T}^3$ is a periodic domain, without loss of generality, in this paper we take $\mathbb{T}^3=[-\pi,\pi]^3.$
\end{theorem}
Furthermore, we also have the following results:
\begin{theorem}\label{main theorem2}
Under the  conditions of Theorem \ref{main theorem1}, then the solution presented
by theorem 1.3 is analytic about the space and time.
%\begin{equation}
%t^{\frac{k}{2}+m}\partial_t^m\nabla^k U   \in \mathbb{X}_{+\infty}.
%\end{equation}
\end{theorem}

\begin{remark}We dot not require $M_0$ to be small in our theorem, actually, it can be as large as you want.
 Up to our knowledge,  it is the first large data results for the incompressible Navier-Stokes equations in $BMO^{-1}$ space.
\end{remark}

\begin{remark}
Our result generalizes the Koch and Tataru result in the space periodic case by taking $\lambda=0$.
\end{remark}

\begin{remark}Our result implies the global nonlinear stability of Beltrami flow for the Navier-Stakes equations in the critical space $BMO^{-1}$. Let $u_0=u_{01}+u_{02}$, with $\int_{\mathbb{T}^3} u_{01}(y)dy=\int_{\mathbb{T}^3} u_{02}(y)dy=0$ and $\nabla\cdot u_{01}=\nabla\cdot u_{02}=0$. We call $u_{01}$ a Beltrami flow initial data if there exists a real number $\lambda$ such that $\nabla\times u_{01}=\lambda u_{01}$, then our Theorem implies that as long as $u_{02}$ is small in $BMO^{-1}$, we can have global existence.

 This kind of data can be constructed as follows (without loss of generality, we consider the case $\lambda >0$),
let
\begin{equation}\label{1.8}
u_{01}=\nabla\times u_{03}+(-\Delta)^{1/2}u_{03}
\end{equation}
with $\nabla\cdot u_{03}=0$.
Then, we get
\begin{equation}\label{1.9}
\nabla\times u_{01}=(-\Delta)^{1/2}u_{01}.
\end{equation}

If we take
\begin{equation}\label{1.10}
u_{03}(x)=\sum_{n\in \mathbb{Z}^3, |n|=\lambda}a_ne^{\sqrt{-1}n\cdot x}
\end{equation}
provided that $\lambda$ is chosen such that the set
${n_1,n_2,n_3\in \mathbb{Z}, n_1^2+n_2^2+n_3^2=\lambda^2}$ is not empty.
Then, we get
$(-\Delta)^{1/2}u_{01}=\lambda u_{01}$.

\end{remark}

Moreover, we remark our result  not only holds for small perturbation of a Beltrami flow. Actually, there exists large class of data that is not a small perturbation of Beltrami flow but still we can prove global existence, as the next Corollary demonstrates:

\begin{corr}\label{main corr 1}
Let $M_0$ be an arbitrary positive constant,  suppose that the initial data
$u_0$ of system \eqref{1.1} with
\begin{equation}\label{1.11}
u_0=u_{01}+u_{02}
\end{equation}
and $u_{01}, u_{02}$ are both periodic functions with
\begin{equation}\label{1.12}
\int_{\mathbb{T}^3} u_{01}(y)dy=\int_{\mathbb{T}^3} u_{02}(y)dy=0, \quad
\nabla\cdot u_{01}=\nabla\cdot u_{02}=0,
\end{equation}
as well as
\begin{equation}\label{1.13}
\|u_{01}\|_{BMO^{-1}(\mathbb{R}^3)}\leq M_0,
\quad \|u_{02}\|_{BMO^{-1}(\mathbb{R}^3)}\leq \varepsilon_1.
\end{equation}

Suppose that
\begin{equation}\label{14}
u_{01}(x)=\sum_{i=1}^N\phi_i(x)
\end{equation}
where $\phi_i$ are Beltrami type data satisfying
\begin{equation}\label{1.15}
\nabla\times \phi_i=\lambda_i\phi_i, \quad \nabla\cdot\phi_i=0
\end{equation}
with
\begin{equation}\label{1.16}
1\le \lambda_1<\lambda_2<\cdots <\lambda_N
\end{equation}
and there exists $0<b<1$, such that
\begin{equation}\label{1.17}
\lambda_N-\lambda_1\le \varepsilon\lambda_1^{1-b}
\end{equation}
then there exists a small constant $\varepsilon(b, M_0)$ which
depend on $b$ and $M_0$ , such that the system \eqref{1.1} admits a unique
 global periodic solution  provided that $\varepsilon\le \varepsilon(b, M_0)$ and $\varepsilon_1\le \varepsilon(b,M_0)\lambda_1^{-b}(1+\lambda_1 )^{-1}$.
\end{corr}
%\begin{equation}\label{1.107}
%\| u_0\|_{L^{2}(\mathbb{T}^3)}\leq \varepsilon_0.
%\end{equation}
%and there exists an absolute constant $\delta$ independent of $\lambda ,M_0$
%\begin{equation}\label{1.108}
%\| \nabla\times u_0-\lambda u_0\|_{BMO^{-2}(\mathbb{R}^3)}\leq \delta.
%\end{equation}
%\end{theorem}
%\begin{remark}By Theorem 1.7, we can easily construct initial data which is not a small perturbation of the Beltrami flow, actually, by the construction in the last remark, we can take
%\begin{equation}
%u_{03}(x)=\sum_{n\in Z^3, ||n|-\lambda|\le \delta^2M_0^{-1}\lambda }a_ne^{\sqrt{-1}n\cdot x}
%\end{equation}
%provided that in addition the initial energy is sufficiently small.
%\end{remark}
\begin{remark}
 Take one more curl to $\phi_i$, we get
\begin{equation}\label{1.18}
-\Delta\phi_i=\nabla\times(\nabla\times\phi_i )=\lambda_i\nabla\times\phi_i=\lambda_i^2\phi_i.
\end{equation}
Thus, $\phi_i$ are the eigenfunction of the Laplacian and $\lambda_i^2$ are the eigenvalues of the Laplacian. Thus, \eqref{14} are noting but a Fourier series expansion.
\end{remark}

Let us sketch a proof of This Corollary. We have
\begin{eqnarray}
&&\| \nabla\times u_0-\lambda_1 u_0\|_{BMO^{-2}(\mathbb{R}^3)}\\\nonumber
&&\leq\| \nabla\times u_{01}-\lambda_1 u_{01}\|_{BMO^{-2}(\mathbb{R}^3)}+\| \nabla\times u_{02}-\lambda_1 u_{02}\|_{BMO^{-2}(\mathbb{R}^3)}.
\end{eqnarray}

Because the spectrum of $u_{01}$ is concentrated near $\lambda_1$, by Bernstein's inequality, we have
\begin{eqnarray}
\| \nabla\times u_{01}-\lambda_1 u_{01}\|_{BMO^{-2}(\mathbb{R}^3)}\lesssim \lambda_1^{-1}\| \nabla\times u_{01}-\lambda_1 u_{01}\|_{BMO^{-1}(\mathbb{R}^3)}\\\nonumber
=\|\sum_{i=1}^N\frac{\lambda_i-\lambda_1}{\lambda_1}\phi_i\|_{BMO^{-1}(\mathbb{R}^3)}\lesssim M_0\varepsilon\lambda_1^{-b}.
\end{eqnarray}
On the other hand, by Lemma 2.3, we have
\begin{eqnarray}
&&\| \nabla\times u_{02}-\lambda_1 u_{02}\|_{BMO^{-2}(\mathbb{R}^3)} \le \| \nabla\times u_{02}\|_{BMO^{-2}(\mathbb{R}^3)}+\lambda_1 \|u_{02}\|_{BMO^{-2}(\mathbb{R}^3)}\\\nonumber
  &&\lesssim \|u_{02}\|_{BMO^{-1}(\mathbb{R}^3)}+\lambda_1\|u_{02}\|_{BMO^{-1}(\mathbb{R}^3)}\lesssim \varepsilon\lambda_1^{-b}.
\end{eqnarray}
Therefore, Corollary 1.8 follows.

%As a matter of fact, it is possible to prove a more general result then this Theorem.
%One can take $u_0=u_{01}+u_{02}$ and such that $u_{02} $ is small in $BMO^{-1}$
%and $u_{01}=\sum_{j=1}^\infty u_{01j}$
%with
%$u_{01j}=\sum_{i=1}^{N_j}\phi_{i}^j$
%where $\phi_{i}^j$ are Beltrami type data satisfying
%\begin{equation}\label{1.24}
%\nabla\times \phi_i^j=\lambda_i^j\phi_i^j, \quad \nabla\cdot\phi_i^j=0
%\end{equation}
%such that
%\begin{equation}\label{1.25}
%1\le |\lambda_1^j|<|\lambda_2^j|<\cdots <|\lambda_{N_j}^j|
%\end{equation}
%and there exists $0<a<1$ such that
%\begin{equation}\label{1.26}
%\sum_{i=1}^{N_j}|\lambda_{i}^j-\lambda_1^j|\le \epsilon(\lambda_1^j)^a\quad \forall j
%\end{equation}
%and
%\begin{equation}\label{1.27}
%\frac{|\lambda_1^j|}{|\lambda_1^{j+1}|}\le \epsilon\quad \forall j
%\end{equation}
%Suppose that
%\begin{equation}\label{1.28}
%\sum_{j=1}^{\infty}\|u_{01j}\|_{BMO^{-1}(\mathbb{R}^3)}\le M_0<\infty
%\end{equation}
%Then there exists an $\epsilon_0 (a, M_0)$ depending only on $a$ and $M_0$ such that system \eqref{1.1} has a global
%space periodic solution provided that $\epsilon\le\epsilon_0(a, M_0)$.
%However, for the convenience of exposition, we will not attempt to prove such a general Theorem.
% Second, we established some new estimates in the scaling invariant space, which plays the key role in our proof (See Proposition 3.4 and Corollary 3.6).

Our paper is organized as follows: in the next section, we present  some preliminaries. In section 3, we shall give some prepare work for the linear   heat equation. The main theorem will be proved in section 4(both Theorem 1.3 and Theorem 1.4).
Throughout this paper, we sometimes use the notation $A\lesssim B$ as an equivalent to $A \leq CB$ with a uniform constant $C$.

%%%%%%%%%%%%%%%%%%%%%%%%%%%%%%%%%%%%%%%%%%%%%%%%%%%%%%%%%%%%%%%%%%%%%%%%%%%%%%%%%%%%%%%%%%
%%%%%%%%%%%%%%%%%%%%%%%%%%%%%%%%%%%%%%%%%%%%%%%%%%%%%%%%%%%%%%%%%%%%%%%%%%%%%%%%%%%%%%%%%%%%%%%%%%%%%%%%%%%%%%%%%%%%%%%%%%%%%%%%%%%%%%%%%

\section{Preliminaries}
At the beginning, we recall  some properties for the  Leary projection operator $\mathbb{P}$ to
divergence free vector fields, which is defined by its matrix valued Fourier
multiplier $\hat{\mathbb{P}}(\xi)=\delta_{ij}-\frac{\xi_i\xi_j}{|\xi|^2}$. For any multi-indices $\alpha$, this symbol satisfies Mihlin-Hormander  condition $\sup\limits_{|\xi|\not=0}|\xi|^\alpha|\partial_{\xi}^{\alpha}\hat{\mathbb{P}}(\xi)|\leq C.$
Besides,  we have the following pointwise bound (see \cite{Lemarie-Rieusset}  Proposition 11.1).
\begin{lemma} Denote $e^{t\Delta}$ as the heat operator, $n$ is the space dimensions, and $ \tilde{\mathbb{P}}(x,t )$  is the kernel of  $\nabla^{k+1}\mathbb{P}e^{t\Delta}$, then there holds
\begin{equation}\label{2.1}
\tilde{\mathbb{P}}(x,t )   \leq C(k)\frac{1}{(\sqrt{t}+|x|)^{n+k+1}},
\end{equation}
where and $C(k)$ is a constant depending on $k$.
\end{lemma}

\begin{lemma} Let $\mathbb{K}(x,t)=\frac{1}{\sqrt{t}^n}e^{-\frac{x^2}{4t}}$, then there exists a polynomial $J^{k+2m}(\frac{x}{\sqrt{t}})$  with degree  $k+2m$, such that
\begin{equation}\label{2.2}
\partial_t^m\nabla^k \mathbb{K}(x,t )= \frac{1}{t^{m+\frac{k}{2}}}\mathbb{K}(x,t )J^{k+2m}(\frac{x}{\sqrt{t}}).
\end{equation}
\begin{proof}
It can be proved by induction.
\end{proof}
\end{lemma}

\begin{lemma} under the assumption that $u_0$ is a periodic function and $\bbint_{\mathbb{T}^3}u_0(y) dy=0$, We have
\begin{equation}\label{2.3}
\|u_0\|_{BMO^{-2}}\lesssim \|u_0\|_{BMO^{-1}}  .
\end{equation}
\begin{proof}

Since $\bbint_{\mathbb{T}^3}u_0(x+y) dy=0$, we have
 \begin{eqnarray} \label{2.4}
&&u_0(x)=\bbint_{\mathbb{T}^3}(u_0(x)-u_0(x+y))dy \\
\nonumber
&&=- \bbint_{\mathbb{T}^3} \int_0^1 \frac{d}{ds}u_0(x+sy) ds dy  =- \nabla_x \bigg(\int^1_0  \bbint_{\mathbb{T}^3} u_0(x+sz) \cdot zdzds\bigg).
\end{eqnarray}
Thus, the conclusion follows
\end{proof}
\end{lemma}

  For completeness, we also  give the following well-known equality:
\begin{lemma} Let
$\nabla\cdot b=0$ and  $\bigtriangledown\cdot h=0$, then we have
\begin{equation}\label{2.5}
b\cdot\bigtriangledown h+h\cdot \bigtriangledown b=-b\times(\bigtriangledown\times h)-h\times(\bigtriangledown\times b)+\bigtriangledown (b\cdot h).
\end{equation}
\end{lemma}

\begin{lemma} \label{a}
Assume that
\begin{equation}
\hat{F}(\xi)=\int_{\mathbb{R}^3} \frac{\hat{g}(\xi-\eta)\hat{h}(\eta)}{|\xi-\eta|+|\eta|} d \eta
\end{equation}
where $\hat{F}$ denotes the Fourier transform of F etc.
then for any $\frac{1}{2}<a<1$, and  $0<\kappa <1-a$ there holds,
\begin{equation}
{\left\| F \right\|}_{B^{\kappa}_{\infty,\infty}} \lesssim {\left\| g \right\|}_{B^{\kappa-1+a}_{\infty,\infty}} {\left\| h \right\|}_{B^{-a}_{\infty,\infty}}
\end{equation}
where ${\left\| F \right\|}_{B^{s}_{\infty,\infty}} \triangleq \sup\limits_{\lambda} {\left\|  \mathcal{P}_{\lambda}F \right\|}_{L^{\infty}} \lambda^s$ with $\mathcal{P}_\lambda$ be the Littlewood-Paley decompositions.
\end{lemma}
\begin{proof}
 We take
a   nonnegative smooth  function  $  \psi $ with
%\begin{equation}
%Supp \ \phi \subset \{ \frac{1}{2} \leq |\xi| \leq 2 \},%\quad 0\leq\phi\leq1,
%\end{equation}
%and
\begin{equation}
Supp \ \psi \subset \{ \frac{1}{4} \leq |\xi| \leq 4\}.%\ \mbox{and} \
%\psi \equiv 1 ,\ \mbox{if} \ \frac{1}{2} \leq |\xi| \leq 2,
\end{equation}
 Moreover, for all $\xi\in \mathbb{R}^3$ , $0\leq \psi\leq 1$ and
\begin{eqnarray}\nonumber
 %\sum_{j\in\mathbb{Z}} \phi(2^{-j}\xi)=
  \sum_{ k\in\mathbb{Z}} \psi(2^{-k}\xi) = 1.
\end{eqnarray}

Write $g_\lambda=\mathcal{P}_\lambda g=\mathcal{F}^{-1}(\psi(\lambda^{-1} \xi) \hat{g})
$ and $h_\mu=\mathcal{P}_\mu h=\mathcal{F}^{-1}(\psi(\mu^{-1} \xi) \hat{h})
$, we have
\begin{equation}
g=\sum_{\lambda=2^j} g_{\lambda},
\end{equation}
and
\begin{equation}
h=\sum_{\mu=2^k } h_{\mu}.
\end{equation}

Thus, we rewrite
\begin{equation}
\hat{F}(\xi)=\sum_{\lambda=2^j, \mu=2^k} \int_{\mathbb{R}^3}\frac{\psi_1({\lambda}^{-1}(\xi-\eta))\psi_1({\mu}^{-1}\eta)}
{|\xi-\eta|+|\eta|}\hat{g_{\lambda}}(\xi-\eta)\hat{h_{\mu}}(\eta)d\eta,
\end{equation}
where $\psi_1\in C_0^{\infty}$ such that $\psi_1(\xi)=1$ in the support of $\psi$. Therefore
\begin{equation}
F(x)=\sum_{\lambda=2^j, \mu=2^k} \int_{\mathbb{R}^3 \times \mathbb{R}^3} \Phi_{\mu\lambda} (x-y,x-z)g_{\lambda}(y)h_{\mu}(z) dy dz
\end{equation}
with
\begin{equation}
\hat{\Phi}_{\mu\lambda}(\xi,\eta)=\frac{\psi_1({\lambda}^{-1} \xi)\psi_1({\mu}^{-1} \eta)}{|\xi|+|\eta|}.
\end{equation}

We then split $F(x)$ as
\begin{eqnarray}
F(x)& =& \sum_{\lambda \leq 2^{-10} \mu} \int_{\mathbb{R}^3 \times \mathbb{R}^3}\Phi_{\mu\lambda} (x-y,x-z)g_{\lambda}(y)h_{\mu}(z) dy dz \\\nonumber
&&+\sum_{\mu \leq 2^{-10} \lambda}\Phi_{\mu\lambda} (x-y,x-z)g_{\lambda}(y)h_{\mu}(z) dy dz \\\nonumber
&&+\sum_{2^{-10} \lambda \leq \mu \leq 2^{10}\lambda }\Phi_{\mu\lambda} (x-y,x-z)g_{\lambda}(y)h_{\mu}(z) dy dz \\\nonumber
&\triangleq&F_1(x)+F_2(x)+F_3(x).
\end{eqnarray}

Thus,
\begin{eqnarray}
{\|F_1(x)\|}_{B^{\kappa}_{\infty,\infty}} & \lesssim& \sup\limits_{\mu} {\mu}^{\kappa} \sum_{\lambda \leq 2^{-10} \mu} {\|\Phi_{\mu\lambda}\|}_{L^1}{\|g_{\lambda}\|}_{L^{\infty}}{\|h_{\mu}\|}_{L^{\infty}}  \\\nonumber
& \lesssim & \sup\limits_{\mu} {\mu}^{\kappa} \left(\sum_{\lambda \leq 2^{-10} \mu} {\|\Phi_{\mu\lambda}\|}_{L^1} {\lambda}^{-\kappa +1-a} {\mu}^{a} \right) {\|g\|}_{B^{\kappa -1+a}_{\infty,\infty}} {\|h\|}_{B^{-a}_{\infty,\infty}}.
\end{eqnarray}

Noting that
\begin{equation}
{\|{\Phi}_{\mu\lambda}\|}_{L^1}={\mu}^{-1} {\|\tilde{{\Phi}}_{\mu\lambda}\|}_{L^1}
\end{equation}
where
\begin{equation}
\hat{\tilde{\Phi}}_{\mu\lambda}=\frac{\psi_1(\xi)\psi_1(\eta)}{{\mu}^{-1}\lambda |\xi|+|\eta|}.
\end{equation}

Obviously any differentiation of $\hat{\tilde{\Phi}}_{\mu\lambda}$ is bounded, therefore $\tilde{{\Phi}}_{\mu\lambda}$ decays in any polynomial.
Thus, ${\|\tilde{\Phi}_{\mu\lambda}\|}_{L^1} \leq C$,
therefore we have
\begin{eqnarray*}
{\|F_1(x)\|}_{B^{\kappa}_{\infty,\infty}} & \leq &
\sup\limits_{\mu} {\mu}^{\kappa} \left( {\sum_{\lambda \leq 2^{-10}\mu}} {{\lambda}^{-\kappa+1  -a} {\mu}^{a-1}} \right) {\|g\|}_{B^{\kappa -1+a}_{\infty,\infty}} {\|h\|}_{B^{-a
}_{\infty,\infty}} \\
& \lesssim &  {\|g\|}_{B^{\kappa -1+a}_{\infty,\infty}} {\|h\|}_{B^{-a}_{\infty,\infty}}.
\end{eqnarray*}

We  can handle $F_2$ in a similar way, here we use the similar fact ${\|{\Phi}_{\mu\lambda}\|}_{L^1}={\lambda}^{-1} {\|\tilde{{\Phi}}^1_{\mu\lambda}\|}_{L^1}\lesssim \lambda^{-1}$ when $\mu\leq 2^{-10}\lambda$.

 To estimate $F_3$,  noting there are finite terms for $2^{-10}\lambda\leq \mu\leq 2^{10}\lambda$ by a given $\mu$, therefore
\begin{eqnarray}
&&{\|F_3(x)\|}_{B^{\kappa}_{\infty,\infty}} \lesssim\sup\limits_{\mu_1} {\mu_1}^{\kappa} \sum_{\mu\ge\mu_1}\sum_{ 2^{-10} \mu\le \lambda \leq 2^{10} \mu} {\|\Phi_{\mu\lambda}\|}_{L^1}{\|g_{\lambda}\|}_{L^{\infty}}{\|h_{\mu}\|}_{L^{\infty}}  \\\nonumber
& & \lesssim \sup\limits_{\mu_1} {\mu_1}^{\kappa}\sum_{\mu\ge \mu_1} \left( \sum_{ 2^{-10} \mu\le \lambda \leq 2^{10} \mu} {\|\Phi_{\mu\lambda}\|}_{L^1} {\lambda}^{-\kappa +1-a} {\mu}^{a} \right) {\|g\|}_{B^{\kappa -1+a}_{\infty,\infty}} {\|h\|}_{B^{-a}_{\infty,\infty}}
\\\nonumber
&& \lesssim \sup\limits_{\mu_1} {\mu_1}^{\kappa} \sum_{\mu\ge \mu_1}\left({\mu}^{-\kappa -a} {\mu}^{a} \right) {\|g\|}_{B^{\kappa -1+a}_{\infty,\infty}} {\|h\|}_{B^{-a}_{\infty,\infty}}\lesssim {\|g\|}_{B^{\kappa -1+a}_{\infty,\infty}} {\|h\|}_{B^{-a}_{\infty,\infty}}.
\end{eqnarray}

\end{proof}
\begin{corr} \label{b}
Let F,g,h are all periodic functions in $\mathbb{T}^3$,
\begin{equation}
g=\sum_n a_n e^{\sqrt{-1}n \cdot x},\ h=\sum_n b_n e^{\sqrt{-1}n \cdot x},\ F=C_n e^{\sqrt{-1}n \cdot x}
\end{equation}
and
\begin{equation}
\int_{\mathbb{T}^3} g dx=\int_{\mathbb{T}^3} h dx=0.
\end{equation}
If $C_n=\sum_{j+k=n} \frac{a_j b_k}{|j|+|k|}$, then regard F,g,h as functions on $\mathbb{R}^3$, for $\frac{1}{2}<a<1$ and $0<\kappa <1-a$,  we have
\begin{equation}
{\| F \|}_{B^{\kappa}_{\infty,\infty}(\mathbb{R}^3)} \lesssim {\| g \|}_{B^{\kappa -1+a}_{\infty,\infty}(\mathbb{R}^3)} {\| h \|}_{B^{-a}_{\infty,\infty}(\mathbb{R}^3)}.
\end{equation}
\end{corr}
\begin{proof}
We have for example
\begin{equation*}
\mathcal{F}(g) = \sum_n a_n \mathcal{F}(e^{\sqrt{-1}n \cdot x}) =\sum_n a_n \delta(-n+\xi),
\end{equation*}
here and hereafter, we use the notations $\mathcal{F}$ and $\mathcal{F}^{-1}$ be the Fourier transform and its inversion respectively.
Therefore, Corollary $\ref{b}$ follows from Lemma $\ref{a}$.
\end{proof}
%\begin{lemma}
%For any function $h(x),x\in T^3$ satisfying
%\begin{equation}
%\int_{T^3}h(x)dx =0,
%\end{equation}
%then we have
%\begin{equation}
%|h|_{B^{-1-s,\infty}_\infty(T^3)}\leq C |h|_{B^{-s}_{BMO}(T^3)}
%\end{equation}
%\end{lemma}
%\begin{proof}
%First, we prove
%\begin{equation}
%|D^{-1}h|_{B^{0,\infty}_\infty}\leq |h|_{B^{0,\infty}_\infty} \quad \mbox{and} \quad |D^{-1}h|_{BMO}\leq C|h|_{BMO},
%\end{equation}
%where  $D= \sqrt{-\triangle}.$

%Since
 %$\int_{T^3}h(x)dx=0$, so $h$ has no zero  mode.  Then $D^{-1}$ can also be regarded as a pseudo-differential operator of order zero, so it maps $B^{0,\infty}_\infty$ to $B^{0,\infty}_\infty$.
%By replacing $h$ by $D^{-s}h$ we can get the desired results.
%\end{proof}

\section{Some prepare work for linear heat equation}
We introduce   the following notations:
\begin{deff}
 Let $g$ be a function defined on $\mathbb{R}^3\times [0,T^*)\ \ (0<T^*\leq +\infty)$,  we say $g\in \mathbb{X}_{T^*} $ if
\begin{eqnarray}\label{3.1}
\|g\|_{ \mathbb{X}_{T^*}}&\triangleq& \sup\limits_{0<t\leq T^*}t^{\frac{ 1}{2} }\| g\|_{L^\infty(\mathbb{R}^3)}
\\\nonumber
&&+\sup\limits_{ 0<  r\leq \sqrt{T^*}}\bigg(r^{-3}\int_{Q(y_0,r)}|  g|^2dy  dt
\bigg)^{\frac12} < +\infty.
\end{eqnarray}
 We say $g\in \mathbb{Z}^d_{T^*} $ if
\begin{eqnarray}\label{3.2}
\|g\|_{ \mathbb{Z}_{T^*}^d}&\triangleq& \sup\limits_{0<t\leq T^*}t^{\frac{ 1-d}{2} }\| g\|_{L^\infty(\mathbb{R}^3)}
\\\nonumber
&&+\sup\limits_{ 0<  r\leq \sqrt{T^*}}\bigg(r^{-(1+2d)}\int_{Q(y_0,r)}|  g|^2dy  dt
\bigg)^{\frac12} < +\infty.
\end{eqnarray}
 We say $g\in \mathbb{Y}_{T^*} $ if
\begin{eqnarray}\label{3.3}
\|g\|_{ \mathbb{Y}_{T^*}}&\triangleq& \sup\limits_{0<t\leq T^*}t\| g\|_{L^\infty(\mathbb{R}^3)}
\\\nonumber
&&+\sup\limits_{ 0<  r\leq \sqrt{T^*}} r^{-3}\int_{Q(y_0,r)}|  g|dy  dt
< +\infty.
\end{eqnarray}

\end{deff}

To go ahead,
we shall give some estimates for the following homogenous heat equation in $\mathbb{R}^3$:
\begin{equation}\label{3.4}
\begin{cases}
\partial_tu-\Delta u=0,\\
 u|_{t=0}=u_0, \quad \nabla \cdot u_0=0.
\end{cases}
\end{equation}
We have the following proposition:
\begin{proposition} \label{proposition 3.1}
There exists a uniform constant $C_0$, such that
\begin{equation}\label{3.5}
\|u \|_{X_{T^*}} \leq C_0 \|u_0\|_{BMO^{-1}} .
\end{equation}
\end{proposition}
\begin{proof}
We write the solution of the linear equation as
\begin{equation}\label{3.6}
u=S(t)u_0,
\end{equation}
where $S(t)$ is the heat flow.
Then   we have:
\begin{eqnarray}\label{3.7}
&&t^{\frac{ 1}{2} }\| u\|_{L^\infty(\mathbb{R}^3)}
= 4 t^{-\frac{1}{2}}\|\int_{\frac{t}{8}}^{\frac{3t}8} S(t-\tau)(S(\tau)u_0)d\tau\|_{L^\infty(\mathbb{R}^3)}
\\\nonumber
&&\lesssim  t^{-\frac{1}{2}}\|\int_{\frac{t}{8}}^{\frac{3t}8}\big(\int_{\mathbb{R}^3} \frac1{\sqrt{t-\tau}^3}e^{-\frac{(y-\tilde{y})^2}{4(t-\tau)}}(e^{\tau\Delta}u_0)^2d\tilde{y}\big)^{1/2}   d\tau\|_{L^\infty(\mathbb{R}^3)}
     \\\nonumber
&&\lesssim t^{-\frac12}\|\big(\int_{\frac{t}{8}}^{\frac{3t}8}\int_{\mathbb{R}^3} \frac1{\sqrt{t-\tau}^3} e^{-\frac{(y-\tilde{y})^2}{4(t-\tau)}}(e^{\tau\Delta}u_0)^2d\tilde{y}d\tau\big)^{1/2}\|_{L^\infty(\mathbb{R}^3)}\cdot \sqrt{t}
\\\nonumber
&&\lesssim\|\big(\sum\limits_{q=0}^\infty  e^{-\frac{q^2}{4}}\frac1{\sqrt{t}^3}\int_{\frac{t}{8}}^{\frac{3t}8}\int_{q\leq\frac{|y-\tilde{y}|}{\sqrt{t}}\leq q+1}
(e^{\tau\Delta}u_0)^2d\tilde{y}d\tau\big)^{1/2}\|_{L^\infty(\mathbb{R}^3)}  \\\nonumber
&&\lesssim     \sup\limits_{y\in \mathbb{R}^3} \big(  \frac1{\sqrt{t}^3}\int_{0}^{t}\int_{B(y,\sqrt{t})}
(e^{\tau\Delta}u_0)^2d\tilde{y}d\tau\big)^{1/2}\lesssim \|u_0\|_{BMO^{-1 }(\mathbb{R}^3)}.
\end{eqnarray}
 By using the  definition $ \sup\limits_{y_0\in  \mathbb{R}^n}  \big(r^{-n}\int_{Q(y_0,r)}u^2dyd\tau\big)^{1/2}=\|u_0\|_{BMO^{-1}}$.
 Thus, we finished the proof of our Theorem.

\end{proof}

\begin{proposition}Let $u$ be the solution of system \eqref{3.4}, then there exists a uniform constant $C_0>0$, such that
\begin{equation}\label{3.5}
\left(\int_0^{+\infty}\|u(t) \|_{L^\infty}^2 dt\right)^{\frac{1}{2}} \leq C_0 \|u_0\|_{B^{-1}_{\infty ,2
}} .
\end{equation}
\end{proposition}
\begin{proof}
See the proof of Lemma 1.5.1 in Chemin \cite{chemin}.
\end{proof}

\begin{proposition} \label{proposition 3.108}
For any $0<d<1$,  there exists a positive constant $C_0=C_0(d)$, such that
\begin{equation}\label{3.8}
\|(-\Delta)^{-d/2}u \|_{Z_{T^*}^d} \leq C_0 \|u_0\|_{BMO^{-1}} .
\end{equation}
\end{proposition}
\begin{proof}
For the simplicity of exposition, we take $d=\frac{1}{2}$, the general case is the same.
Similar to the proof of our previous Proposition, we can get
\begin{eqnarray}\label{3.9}
&&t^{\frac{ 1}{4} }\| (-\Delta)^{-1/4}u\|_{L^\infty(\mathbb{R}^3)}\\\nonumber
&&\lesssim     \sup\limits_{y\in \mathbb{R}^3} \big(  \frac1{t^2}\int_{0}^{t}\int_{B(y,\sqrt{t})}
((-\Delta)^{-1/4}u(\tau, \tilde{y}))^2d\tilde{y}d\tau\big)^{1/2}.
\end{eqnarray}
Thus, it remains to prove
\begin{equation}\label{3.10}
 \sup\limits_{y_0\in  \mathbb{R}^3}    \big(r^{-4}\int_{Q(y_0,r)} ((-\Delta)^{-1/4}u)^2dydt\big)^{1/2}
\lesssim \|u_0\|_{BMO^{-1 }(\mathbb{R}^3)}.
\end{equation}

By definition, we have
\begin{equation}\label{3.11}
\| (-\Delta)^{-1/4}u\|_{L^\infty(\mathbb{R}^3)}
= 4 r^{-2}\|\int_{\frac{r^2}{8}}^{\frac{3r^2}8} (-\Delta)^{-1/4}S(t-\tau)(S(\tau)u_0)d\tau\|_{L^\infty(\mathbb{R}^3)}.
\end{equation}

Denote $K_1(x)$ be the kernel of the operator $(-\Delta)^{-1/4}S(t-\tau)$, then we have
\begin{eqnarray}\label{3.12}
|K_1(x)|&=&\bigg|\mathcal{F}^{-1}\bigg(\mathcal{F}\big((-\Delta)^{-1/4}S(t-\tau)\big)\bigg)\bigg|\\\nonumber
&=&
 \bigg|\mathcal{F}^{-1}  \big(|\xi|^{-\frac12}e^{-\xi^2t}\big) \bigg|=\bigg|\int_{\mathbb{R}^3}e^{ix\cdot\xi}|\xi|^{-\frac12}e^{-\xi^2t}d\xi \bigg|
 \\\nonumber
&=&
 e^{-\frac{x^2}{4t}}\bigg|\int_{\mathbb{R}^3}e^{-(\sqrt{t}\xi-i\frac{x}{2t})^2}|\xi|^{-\frac12} d\xi \bigg|
\\\nonumber
&\lesssim & \frac{t^{1/4}}{\sqrt{t}^3}  e^{-\frac{x^2}{4t}}\bigg|\int_{\mathbb{S}^2}\int_0^\infty e^{-(\sqrt{t}\xi-i\frac{x}{2t})^2} |\sqrt{t}\xi|^{\frac32} d(\sqrt{t}|\xi|) \bigg|
 \lesssim   \frac{t^{1/4}}{\sqrt{t}^3}  e^{-\frac{x^2}{4t}}
.
\end{eqnarray}

Then by using \eqref{3.12} and similar to \eqref{3.7},  we have
\begin{eqnarray}\label{3.13}
&& | (-\Delta)^{-1/4}u |
=  4 r^{-2} |\int_{\frac{r^2}{8}}^{\frac{3r^2}8} (-\Delta)^{-1/4}S(t-\tau)(S(\tau)u_0)d\tau |
\\\nonumber
&&\lesssim  r^{-2}\|\int_{\frac{r^2}{8}}^{\frac{3r^2}8}\big(\int_{\mathbb{R}^3}K_1(x-y) d\tilde{y}\big)^{1/2} \big(\int_{\mathbb{R}^3}K_1(x-y)(e^{\tau\Delta}u_0)^2d\tilde{y}\big)^{1/2}   d\tau\|_{L^\infty(\mathbb{R}^3)}
     \\\nonumber
&&\lesssim r^{-1+\frac12}\|\big(\int_{\frac{r^2}{8}}^{\frac{3r^2}8}\int_{\mathbb{R}^3} \frac1{\sqrt{t-\tau}^3} e^{-\frac{(y-\tilde{y})^2}{4(t-\tau)}}(e^{\tau\Delta}u_0)^2d\tilde{y}d\tau\big)^{1/2}\|_{L^\infty(\mathbb{R}^3)}
\\\nonumber
&&\lesssim r^{-\frac12}\|u_0\|_{BMO^{-1 }(\mathbb{R}^3)}.
\end{eqnarray}
From \eqref{3.13} and noting \eqref{3.10},  our proposition follows.
\end{proof}

\begin{corr}under the assumption that $u_0$ is a periodic function and  $\bbint_{\mathbb{T}^3}u_0(y) dy=0$, then for $0<d<1$, we have
\begin{equation}\label{3.14}
\|(-\Delta)^{-(\frac{2-d}{2})}\nabla\times u_0\|_{BMO^{-1}}\lesssim \|u_0\|_{BMO^{-1}}  .
\end{equation}
\begin{proof}
Without loss of generality, we take $d=\frac{1}{2}$, the general case is the same.

Denote the matrix function $K_2(x)$ be the kernel of the operator $ e^{t\Delta}(-\Delta)^{-3/4}\nabla\times$. We need to estimate the  bound of ${(K_2)}_{ij}(x), i,j=1,\cdots,3.$

Firstly, we have get the matrix Fourier multiplier of the operator $e^{t\Delta}(-\Delta)^{-3/4}\nabla\times$ as
\begin{equation}\label{3.15}
e^{-\xi^2t}|\xi|^{-\frac32}\bigg(\begin{array}{ccc}
  0& -i\xi_3 &i\xi_2 \\
  i\xi_2 & 0 & -i\xi_1 \\
  -i\xi_2 & i\xi_1 &0
\end{array}
 \bigg)\triangleq e^{-\xi^2t}|\xi|^{-\frac32}A(\xi).
\end{equation}

Then for $i,j=1,\cdots, 3$, we have
\begin{eqnarray}\label{3.16}
&&|{(K_2)}_{ij}(x)|=\big|\mathcal{F}^{-1}\big(e^{-\xi^2t}|\xi|^{-\frac32}A_{ij}(\xi)\big) \big| =\bigg|\int_{\mathbb{R}^3}e^{ix\cdot\xi}e^{-\xi^2t}|\xi|^{-\frac32}A_{ij}(\xi)d\xi \bigg|
 \\\nonumber
&&\lesssim \frac{t^{1/4}}{\sqrt{t}^3}  e^{-\frac{x^2}{4t}}\bigg|\int_{\mathbb{S}^2}\int_0^\infty e^{-(\sqrt{t}\xi-i\frac{x}{2t})^2} |\sqrt{t}\xi|^{\frac12}A(\sqrt{t}\xi) d(\sqrt{t}|\xi|) \bigg|
 \lesssim   \frac{t^{1/4}}{\sqrt{t}^3}  e^{-\frac{x^2}{4t}}
.
\end{eqnarray}

Denote the matrix function $K_3(x)$ be the kernel of the operator $\nabla(-\Delta)^{-3/4}\nabla\times e^{t\Delta}$, for $i,j=1,\cdots, 3$, we have
\begin{eqnarray}\label{3.17}
&&|(K_3)(x)|=\big|\mathcal{F}^{-1}\big(e^{-\xi^2t}|\xi|^{-\frac12}B(\xi)\big) \big| =\bigg|\int_{\mathbb{R}^3}e^{ix\cdot\xi}e^{-\xi^2t}|\xi|^{-\frac12}B(\xi)d\xi \bigg|
 \\\nonumber
&&\lesssim \frac{t^{-1/4}}{\sqrt{t}^3}  e^{-\frac{x^2}{4t}}\bigg|\int_{\mathbb{S}^2}\int_0^\infty e^{-(\sqrt{t}\xi-i\frac{x}{2t})^2} |\sqrt{t}\xi|^{\frac32}B(\sqrt{t}\xi) d(\sqrt{t}|\xi|) \bigg|
  \lesssim   \frac{t^{-1/4}}{\sqrt{t}^3}  e^{-\frac{x^2}{4t}},
\end{eqnarray}
where we denote $B(\xi)=\frac{\xi}{|\xi|}\otimes A(\xi)$.

Similarly to \eqref{3.13}, we have
\begin{eqnarray}\label{3.18}
&& |(-\Delta)^{-3/4}\nabla\times e^{t\Delta} u_0 |  \\\nonumber
&&=4 r^{-2} |\int_{\frac{r^2}{8}}^{\frac{3r^2}8} (-\Delta)^{-3/4}\nabla\times S(t-\tau)(S(\tau)u_0)d\tau | \\\nonumber
&&\lesssim  r^{-2} |\int_{\frac{r^2}{8}}^{\frac{3r^2}8}\big(\int_{\mathbb{R}^3}K_2(x-y) d\tilde{y}\big)^{1/2} \big(\int_{\mathbb{R}^3}K_2(x-y)(e^{\tau\Delta}u_0)^2d\tilde{y}\big)^{1/2}   d\tau |
     \\\nonumber
&&\lesssim r^{-1+\frac12} |\big(\int_{\frac{r^2}{8}}^{\frac{3r^2}8}\int_{\mathbb{R}^3} \frac1{\sqrt{t-\tau}^3} e^{-\frac{(y-\tilde{y})^2}{4(t-\tau)}}(e^{\tau\Delta}u_0)^2d\tilde{y}d\tau\big)^{1/2} |
\\\nonumber
&&\lesssim r^{-\frac12}\|u_0\|_{BMO^{-1 }(\mathbb{R}^3)}.
\end{eqnarray}
Let $u_0=(\nabla\times)^{-1} w_0$, then we also  have
\begin{eqnarray}\label{3.19}
&& |(-\Delta)^{-3/4}\nabla\times e^{t\Delta} u_0 | \\\nonumber
&&=4 r^{-2} |\int_{\frac{r^2}{8}}^{\frac{3r^2}8} \nabla (-\Delta)^{-3/4}\nabla\times S(t-\tau)(S(\tau)w_0)d\tau | \\\nonumber
&&\lesssim  r^{-2-\frac14} |\int_{\frac{r^2}{8}}^{\frac{3r^2}8} \big(\int_{\mathbb{R}^3}K_3(x-y)(e^{\tau\Delta}w_0)^2d\tilde{y}\big)^{1/2}   d\tau |
     \\\nonumber
&&\lesssim r^{-1-\frac12} |\big(\int_{\frac{r^2}{8}}^{\frac{3r^2}8}\int_{\mathbb{R}^3} \frac1{\sqrt{t-\tau}^3} e^{-\frac{(y-\tilde{y})^2}{4(t-\tau)}}( e^{\tau\Delta}w_0)^2d\tilde{y}d\tau\big)^{1/2} |
\\\nonumber
&&\lesssim r^{-\frac32}\|w_0\|_{BMO^{-1}(\mathbb{R}^3)}.
\end{eqnarray}
Take $\inf$ with respect to $w_0$, we get
\begin{equation}\label{3.20}
 |(-\Delta)^{-3/4}\nabla\times e^{t\Delta} u_0 | \lesssim r^{-\frac32}\|u_0\|_{BMO^{-2}(\mathbb{R}^3)}.
\end{equation}

Combing \eqref{3.18} and \eqref{3.20}, we have
\begin{eqnarray}\label{3.21}
&&|(-\Delta)^{-3/4}\nabla\times e^{t\Delta} u_0|\\\nonumber
 &&\lesssim r^{-1}\|u_0\|_{BMO^{-1 }(\mathbb{R}^3)}^{\frac12}\|u_0\|_{BMO^{-2 }(\mathbb{R}^3)}^{\frac12} \lesssim r^{-1}\|u_0\|_{BMO^{-1 }(\mathbb{R}^3)}.
\end{eqnarray}
The last inequality given above was from  lemma 2.3.

By definition
\begin{equation}\label{3.21}
\|(-\Delta)^{-3/4}\nabla\times u_0\|_{BMO^{-1}}=
\sup\limits_{y_0\in  \mathbb{R}^3}    \big(r^{-3}\int_{Q(y_0,r)} ((-\Delta)^{-3/4}\nabla\times e^{t\Delta} u_0)^2dydt\big)^{1/2},
\end{equation}
then the conclusion follows from  \eqref{3.20}.
\end{proof}
\end{corr}
%%%%%%%%%%%%%%%%%%%%%%%%%%%%%%%%%%%%%%%%%%%%%%%%%%%%%%%%%%%%%%%%%%%%%%%%%%%%%%%%%%%%%%%
The following proposition appear in Koch and Taturu, for the convenience of the reader, we include a slightly different  proof here.

\begin{proposition}\label{prop3.3}
Let  $G_1$ be a tensor, and  $V_1$ be the solution of the following system,
\begin{equation}\label{3.22}
\begin{cases}
V_{1 t}-\triangle V_1+\nabla P_1=\nabla \cdot   G_1,\\
\nabla\cdot V_1=0,\\
t=0: V_1=0,
\end{cases}
\end{equation}
then we have
\begin{equation}\label{3.23}
\|V_1\|_{X_{T^*}}\lesssim \|G_1\|_{Y_{T^*}}.
\end{equation}
\end{proposition}
\begin{proof}
 First, we rewrite the system \eqref{3.22} as an integral equation
 \begin{equation}\label{3.24}
V_1(t)= \int_0^t S(t-\tau)\mathbb{P}\nabla\cdot  G_1 d\tau.
\end{equation}
When $0\leq \tau\leq t/2$, by using Lemma 2.1, we get:
\begin{eqnarray}\label{3.25}
 &&
t^{\frac{1}{2} }
\| \int_0^{t/2} S(t-\tau)\mathbb{P}\nabla \cdot G_1 d\tau\|_{L^\infty(\mathbb{R}^n)}\\ \nonumber
&\lesssim&
t^{\frac{1}{2}}
\|\int_0^{t/2}\int_{\mathbb{R}^n}\frac{1}{(\sqrt{t-\tau}+|y-\tilde{y}|)^{n+1}}G_1d\tilde{y}d\tau\|_{L^\infty(\mathbb{R}^n)}\\ \nonumber
%&\lesssim&
%t^{\frac12}
%\|\int_0^{t/2}\int_{\mathbb{R}^n}\frac{1}{(\sqrt{t-\tau}+|y-\tilde{y}|)^{n+1}}(|\tilde{u}|^2+|\nabla\tilde{d}|^2)d\tilde{y}d\tau\|_{L^\infty(\mathbb{R}^n)}\\\nonumber
%&\lesssim&
%t^{\frac12}
%\|\int_0^{t/2}\int_{\mathbb{R}^n}\frac{1}{\sqrt{t-\tau}^{n+1}(1+\frac{|y-\tilde{y}|}{\sqrt{t-\tau}})^{n+1}}(|\tilde{u}|^2+|\nabla\tilde{d}|^2)d\tilde{y}d\tau\|_{L^\infty(\mathbb{R}^n)}\\\nonumber
&\lesssim&
\|\frac{1}{\sqrt{t}^{n}}\int_0^{t/2}\int_{\mathbb{R}^n}\frac{1}{(1+\frac{|y-\tilde{y}|}{\sqrt{t-\tau}})^{n+1}}|G_1|d\tilde{y}d\tau\|_{L^\infty(\mathbb{R}^n)}\\\nonumber
&\lesssim&
\|\frac{1}{\sqrt{t}^{n}}\sum\limits_{q=0}^{+\infty}\int_0^{t/2}\int_{q\leq\frac{|y-\tilde{y}|}{\sqrt{t-\tau}}\leq q+1}\frac{|G_1|}{(1+q)^{n+1}}d\tilde{y}d\tau\|_{L^\infty(\mathbb{R}^n)}\\\nonumber
&\lesssim&
\|\sum\limits_{q=0}^{+\infty}\frac{q^{n-1}}{(1+q)^{n+1}}\frac{1}{\sqrt{t}^{n}}\int_0^{t/2}\int_{B(y,\sqrt{t})}|G_1|d\tilde{y}d\tau\|_{L^\infty(\mathbb{R}^n)}
\lesssim \|G_1\|_{Y_{T^*}}.
\end{eqnarray}
If $\frac{t}{2}\leq \tau\leq t$,  we get
\begin{equation}\label{3.26}
| S(t-\tau)\mathbb{P}\nabla\cdot G_1 |
\lesssim|\int_{\mathbb{R}^{n}}\frac{1}{(\sqrt{t-\tau}+|y-\tilde{y}|)^{n+1}} d\tilde{y}| \| G_1\|_{L^\infty(\mathbb{R}^n)}
%&&\lesssim
%\frac{1}{\sqrt{t-\tau}}\sum\limits_{l=0}^kC_k^l  \|(|\nabla^{k-l} \tilde{u}||\nabla^l\tilde{u}|+|\nabla^{k-l}\nabla \tilde{d}||\nabla^l \nabla \tilde{d}|)\|_{L^\infty(\mathbb{R}^n)}
% \\\nonumber
\lesssim
\frac{\|G_1\|_{L^\infty(\mathbb{R}^n)}}{\sqrt{t-\tau}}.
\end{equation}
Therefore we have
\begin{eqnarray}\label{3.27}
 &&t^{\frac{1}{2}}
\|\int_{t/2}^t S(t-\tau)\mathbb{P}\nabla \cdot G_1d\tau\|_{L^\infty(\mathbb{R}^n)}
\\\nonumber
&&\lesssim
t^{\frac{1}{2}}
\int_{t/2}^t\frac1{\sqrt{t-\tau}}d\tau  \| G_1\|_{L^\infty(\mathbb{R}^n)}\lesssim
t\| G_1\|_{L^\infty(\mathbb{R}^n)}.
\end{eqnarray}

We still need to estimate the term $ \sup\limits_{y_0\in  \mathbb{R}^n, r>0} \frac{1}{r^n}\int^{r^2}_0\int _{B(y_0,r)}|V_1(s,y)|^2dyds$.
For any given $y_0\in \mathbb{R}^n$ and $r>0$, take a smooth cut-off function
\begin{equation}\label{3.28}
\chi(\frac{y}{r})=
\begin{cases}
1,\ |y-y_0|\leq 3r;\\
0,\ |y-y_0|\geq5r.
\end{cases}
\end{equation}
Then, it is sufficient to estimate the following term
\begin{multline}\label{3.29}
I=  \sup\limits_{y_0\in\mathbb{R}^n, r>0 }   \bigg(r^{-n}\int_{Q(y_0,r)} \big(\nabla \int_0^tS(t-\tau)\mathbb{P}
\chi G_1 d\tau\big)^2dy dt\bigg)^{\frac12}\\
+  \sup\limits_{y_0\in\mathbb{R}^n, r>0 }  \bigg(r^{-n}\int_{Q(y_0,r)} (\nabla  \int_0^t S(t-\tau)\mathbb{P} (1-\chi) G_1
d\tau\big)^2dy dt\bigg)^{\frac12}
\triangleq I_{1}+I_{2}.
\end{multline}

To deal with $I_{1}$,
we will drop the projector $\mathbb{P}$, which is a bounded operator in $L^2$ and
which commutes with $S(t-\tau)$, $\nabla  $ and the integral  about $t$.
We set up a heat function
\begin{equation}\label{3.30}
W_t-\Delta W = \chi G_1, \qquad W|_{t=0}=0.
\end{equation}
Then $I_{1}= \sup\limits_{y_0\in\mathbb{R}^n, r>0}   \bigg(r^{-n}\| \nabla W\|_{L^2(Q(y_0,r))}^2\bigg)^{\frac12} $. We get
\begin{eqnarray}\label{3.31}
&&\| \nabla W\|_{L^2(Q(y_0,r))}^2\leqslant\int_0^{r^2}\int_{\mathbb{R}^n} (\nabla W)^{2}dyd\tau\\\nonumber
&&\lesssim \int_0^{r^2}\int_{\mathbb{R}^n}| \chi G_1 \ W|dy d\tau.
\end{eqnarray}
For the term $\int_0^{r^2}\int_{\mathbb{R}^n}| \chi G_1 W|dy dt$,  recalling that $ \chi=\chi(y/r)$ and $t\leq r^2$, we have
\begin{equation}\label{3.32}
 \int_0^{r^2}\int_{\mathbb{R}^n}| \chi G_1 W|dy d\tau
\lesssim
 \| W\|_{L^\infty_{Q(y_0,5r)}}
  \int_0^{r^2}\int_{\mathbb{R}^n}| \chi G_1 |dy d\tau,
  \end{equation}
  where
\begin{equation}\label{3.33}
  \| W\|_{L^\infty_{Q(y_0,5r)}}
=
\| \int_0^{\tau/2}   S(\tau-s) \chi G_1 ds
+   \int_{\tau/2}^\tau S(\tau-s) \chi G_1 ds\|_{L^\infty_{Q(y_0,5r)}}.
\end{equation}

We shall study the above inequality separately.

By Lemma 2.2, we have
\begin{eqnarray}\label{3.34}
&&\| \int_0^{\tau/2} S(\tau-s) G_1 ds \|_{L^\infty_{Q(y_0,5r)}}             \\ \nonumber
 \\\nonumber
 &&\lesssim\|\frac{1}{\sqrt{\tau}^{n}}\int_0^{\tau/2}\sum\limits_{q=0}^\infty\int_{q\leq\frac{|y-\tilde{y}|}{\sqrt{\tau}}\leq q+1}e^{-q^2} G_1 d\tilde{y}ds\|_{L^\infty_{Q(y_0,5r)}}
  \\\nonumber
% &&\lesssim\|\sum\limits_{m=0}^\infty  m^{n-1}J^l(m)e^{-m^2}\frac{1}{\sqrt{\tau}^{n}}\int_0^{\tau/2}\int_{B(y,\sqrt{t})}|\digamma | d\tilde{y}ds\|_{L^\infty_{Q(y_0,5r)}}
%   \\\nonumber
&& \lesssim\sup\limits_{y\in \mathbb{R}^n, 0<\tau \leq   {T^*}}\frac{1}{\sqrt{\tau}^{n}}\int_0^{\tau}\int_{B(y,\sqrt{\tau})}|G_1| d\tilde{y}ds
\lesssim
\|G_1\|_{Y_{T^*}}.
\end{eqnarray}
Recalling the definition of the cut-off function $\chi$, we have
\begin{eqnarray}\label{3.35}
&&  \| \int_{\tau/2}^\tau\int_{\mathbb{R}^n}\frac{1}{\sqrt{\tau-s}^n}e^{-\frac{(y-\tilde{y})^2}{4(\tau-s)}}  \chi G_1d\tilde{y}ds\|_{L^\infty_{Q(y_0,5r)}}
 \\\nonumber
%&& \lesssim\|\int_{\tau/2}^\tau\int_{\mathbb{R}^n}\frac{1}{\sqrt{\tau-s}^n}e^{-\frac{(y-\tilde{y})^2}{4(\tau-s)}} \sum\limits_{l_1=0}^ls^{\frac{l_1}{2}}|\nabla^{l_1}\digamma| d\tilde{y}ds\|_{L^\infty_{Q(y_0,5r)}}
% \\\nonumber
&& \lesssim \int_{\tau/2}^\tau\int_{\mathbb{R}^n}\frac{1}{\sqrt{\tau-s}^n}e^{-\frac{(y-\tilde{y})^2}{4(\tau-s)}} s^{-1}d\tilde{y}ds \cdot \tau \|G_1\|_{L^\infty_{Q(y_0,5r)}}
 \\\nonumber
&& \lesssim\|G_1\|_{Y_{T^*}}\int_{\tau/2}^\tau \frac1s ds \lesssim \|G_1\|_{Y_{T^*}}.
\end{eqnarray}
To finish the estimate of $I_{1}$, from \eqref{3.32}, we still need to estimate $  r^{-n}\int_0^{r^2}\int_{\mathbb{R}^n}|\chi G_1 |dy d\tau$.
\begin{equation}\label{3.36}
 r^{-n} \int_0^{r^2}\int_{\mathbb{R}^n}|\chi G_1 |dy d\tau
\lesssim
   r^{-n} \int_0^{r^2}\int_{B(y_0,5r)}|G_1| dy d\tau \lesssim \|G_1\|_{Y_{T^*}}.
\end{equation}

Combining \eqref{3.31}-\eqref{3.36}, we get
\begin{equation}\label{3.37}
I_{1}\lesssim \|G_1\|_{Y_{T^*}}.
\end{equation}

As to the term $I_{2}$,  we have
\begin{equation}\label{3.38}
I_{2}^2\leq \sup_{0<r\leq \sqrt{T^*}} r^2\| \nabla \int_0^tS(t-\tau)\mathbb{P}(1-\chi)G_1  d\tau\|_{L^\infty(Q(y_0,r))}^2.
\end{equation}

When $0\leq \tau\leq t\leq (r/2)^2$, noting the cut-off function $1-\chi$ and Lemma 2.1, we have
\begin{eqnarray}\label{3.39}
&& \sup\limits_{y_0\in \mathbb{R}^n,0<r\leq \sqrt{ {T^*}}}   r^2
\| \nabla \int_0^tS(t-\tau)\mathbb{P}(1-\chi)G_1 d\tau\|_{L^\infty(Q(y_0,r))}^2\\\nonumber
&&\lesssim \sup\limits_{y_0\in \mathbb{R}^n,0<r\leq \sqrt{ {T^*}}}    r^{2}
\|\int_0^t \int_{\mathbb{R}^n}\frac{1-\chi(\frac{\tilde{y}}{r})}{(\sqrt{t-\tau}+|y-\tilde{y}|)^{n+1}}|G_1|d\tilde{y}d\tau\|_{L^\infty(Q(y_0,r))}^2 \\
\nonumber
&&\lesssim  \sup\limits_{y_0\in \mathbb{R}^n,0<r\leq \sqrt{ {T^*}}}     r^{2}
\|\int_0^t \sum\limits_{q=1}^\infty\int_{qr\leq|y-\tilde{y}|\leq (q+1)r}\frac{|G_1|}{(qr)^{n+1}}d\tilde{y}d\tau\|_{L^\infty(Q(y_0,r))}^2 \\
\nonumber
&&\lesssim  \sup\limits_{y_0\in \mathbb{R}^n,0<r\leq \sqrt{ {T^*}} }    \big(
\frac{1}{r^n}\int_0^{r^2} \int_{B(y,r)}| G_1| d\tilde{y}d\tau\big)^2
\lesssim \|G_1\|_{Y_{T^*}}^2.
\end{eqnarray}
For the remaining  part $(r/2)^2<t<r^2$,  we shall divide into two parts to estimate.
%\begin{equation}
%I_{22}^2\lesssim
%\sup\limits_{y_0,r>0}r^{2}
%\|t^{\frac{k}{2}}\nabla^{k+1}\big(\int_0^{t/2}+\int_{t/2}^t\big)S(t-\tau)\mathbb{P}(1-\chi)\digamma d\tau\|_{L^\infty(Q(y_0,r))}^2.
%\end{equation}

 (i): when $(r/2)^2<t<r^2$  and $0\leq \tau\leq t/2$, by Lemma 2.1, we have
\begin{eqnarray}\label{3.40}
 && \sup\limits_{y_0\in \mathbb{R}^n,0<r\leq \sqrt{ {T^*}}}   r^{2}
\|  \nabla\cdot \int_0^{t/2} S(t-\tau)\mathbb{P}(1-\chi)G_1 d\tau\|_{L^\infty(Q(y_0,r))}^2
\\\nonumber
&&\lesssim
\sup\limits_{y_0\in \mathbb{R}^n,0<r\leq \sqrt{ {T^*}}}     r^{2}
\| \int_0^{t/2}\int_{\mathbb{R}^n}\frac{(1-\chi)}{(\sqrt{t-\tau}+|y-\tilde{y}|)^{n+1}}|G_1 (\tau, \tilde{y})| d\tilde{y}d\tau\|_{L^\infty(Q(y_0,r))}^2 \\\nonumber
&&\lesssim
\sup\limits_{y_0\in \mathbb{R}^n,0<r\leq \sqrt{ {T^*}}}
\|\frac{1}{\sqrt{t}^{n}}\int_0^{t/2}\sum\limits_{q=0}^\infty\int_{q\leq\frac{|y-\tilde{y}|}{\sqrt{t}}\leq q+1}\frac{|G_1(\tau, \tilde{y})| }{(1+q)^{n+1}}d\tilde{y}d\tau\|_{L^\infty(Q(y_0,r))}^2  \\\nonumber
&&\lesssim
\sup\limits_{y_0\in \mathbb{R}^n,0<r\leq \sqrt{ {T^*}}}
\|\frac{1}{\sqrt{t}^{n}}\int_0^{t/2}\sum\limits_{q=0}^\infty    \frac{ q^{n-1}}{(1+q)^{n+1}}    \int_{B(y,\sqrt{t})}|G_1(\tau, \tilde{y})|     d\tilde{y}d\tau\|_{L^\infty(Q(y_0,r))}^2   \\\nonumber
&&\lesssim \|G_1\|_{Y_{T^*}}^2.
\end{eqnarray}

(ii):  when $(r/2)^2<t<r^2$ and  $t/2\leq \tau \leq t $,  we have
\begin{equation}\label{3.41}
|r  (1-\chi)G_1  |
 %\lesssim \sum\limits_{l=0}^{k-1} r^{k+1}|\nabla^{k-l}\chi \nabla^l\digamma| +|(1-\chi)\tau^{\frac{k+1}{2}}\nabla^k\digamma |\\\nonumber
 \lesssim  \sqrt{\tau} |G_1|\leq\frac{1}{\sqrt{\tau}} \|G_1\|_{Y_{T^*}}.
\end{equation}
Then by \eqref{3.41} and Lemma 2.2, we get
\begin{eqnarray}\label{3.42}
 &&
\sup\limits_{y_0,r>0}r^{2}
\| \nabla\cdot \int_{t/2}^tS(t-\tau)\mathbb{P}(1-\chi)G_1 d\tau\|_{L^\infty(Q(y_0,r))}^2
\\\nonumber
&&\lesssim\|G_1\|_{Y_{T^*}}^2\bigg(
\int_{t/2}^t \frac{1}{\sqrt{\tau}}\int_{\mathbb{R}^n}\frac{1}{(\sqrt{t-\tau}+|\tilde{y}|)^{n+1}}  d\tilde{y}d\tau\bigg)^2 \\\nonumber
&&\lesssim \|G_1\|_{Y_{T^*}}^2\bigg(
\int_{t/2}^t \frac{1}{\sqrt{\tau}}\int_{\mathbb{R}^n}\frac{1}{\sqrt{t-\tau}^{n+1}(1+\frac{|\tilde{y}|}{\sqrt{t-\tau}})^{n+1}}  d\tilde{y}d\tau\bigg)^2 \\\nonumber
&&\lesssim \|G_1\|_{Y_{T^*}}^2\bigg(
\int_{t/2}^t \frac{1}{\sqrt{\tau}}     \frac{1}{\sqrt{t-\tau}}    d\tau      \int_{\mathbb{R}^n}   \frac{1}{(1+|\bar{y}|)^{n+1}}  d\bar{y} \bigg)^2  \lesssim \|G_1\|_{Y_{T^*}}^2.
\end{eqnarray}
Then from \eqref{3.38}-\eqref{3.42},  we can get
 \begin{equation}\label{3.43}
 I_{2}\lesssim \|G_1\|_{Y_{T^*}}^2.
 \end{equation}

\end{proof}

\begin{corr}Let  $G_1$ be a tensor, and  $V_1$ be the solution of the  system \eqref{3.22},
%\begin{equation}\label{3.22}
%\begin{cases}
%V_{1 t}-\triangle V_1+\nabla P_1=\nabla \cdot   G_1,\\
%\nabla\cdot V_1=0,\\
%t=0: V_1=0,
%\end{cases}
%\end{equation}
then for any $0<a<1$, we have
\begin{equation}\label{3.46}
\sup_{0<t<T}t^{\frac{1-a}{2}}\|(-\Delta)^{-a/2}V_1(t, \cdot)\|_{L^\infty}\lesssim \|G_1\|_{Y_{T}}.
\end{equation}
\end{corr}

\begin{proof}

Denote  $K_4(x)$  and $K_5(x)$ be the kernels of the operator $ e^{t\Delta}(-\Delta)^{-a/2}$ and $  e^{\frac{t}{2}\Delta}(-\Delta)^{-a/2}$.
Similarly to \eqref{3.12} and \eqref{3.15}, we have
\begin{equation}\label{3.47}
|{K_4} (x)|=\big|\mathcal{F}^{-1}\big(e^{-\xi^2t}|\xi|^{-a} \big) \big| %=\bigg|\int_{\mathbb{R}^3}e^{ix\cdot\xi}e^{-\xi^2t}|\xi|^{-a} d\xi \bigg|
% \\\nonumber
%&&\lesssim \frac{t^{a/2}}{\sqrt{t}^3}  e^{-\frac{x^2}{4t}}\bigg|\int_{\mathbb{S}^2}\int_0^\infty e^{-(\sqrt{t}\xi+i\frac{x}{2\sqrt{t}})^2} |\sqrt{t}\xi|^{2-a}  d(\sqrt{t}|\xi|) \bigg|
 \lesssim   \frac{t^{a/2}}{\sqrt{t}^3}  e^{-\frac{x^2}{4t}}
,\quad \mbox{and} \quad {K_5} (x)|
%\lesssim \frac{t^{\frac{a}{2}}}{\sqrt{t}^3}  e^{-\frac{x^2}{2t}}\bigg|\int_{\mathbb{S}^2}\int_0^\infty e^{-(\sqrt{\frac{t}2}\xi+i\frac{x}{\sqrt{2t}})^2} |\sqrt{t}\xi|^{2-a}  d(\sqrt{t}|\xi|) \bigg|
 \lesssim   \frac{t^{\frac{a}{2}}}{\sqrt{t}^3}  e^{-\frac{x^2}{2t}}
.
\end{equation}
%and
%\begin{eqnarray}\label{3.47}
%&&|{K_5} (x)|
%\lesssim \frac{t^{\frac{a}{2}}}{\sqrt{t}^3}  e^{-\frac{x^2}{2t}}\bigg|\int_{\mathbb{S}^2}\int_0^\infty e^{-(\sqrt{\frac{t}2}\xi+i\frac{x}{\sqrt{2t}})^2} |\sqrt{t}\xi|^{2-a}  d(\sqrt{t}|\xi|) \bigg|
% \lesssim   \frac{t^{\frac{a}{2}}}{\sqrt{t}^3}  e^{-\frac{x^2}{2t}}
%.
%\end{eqnarray}

By using \eqref{3.22}, we have
\begin{eqnarray}\label{3.48}
&&\|(-\Delta)^{-\frac{a}{2}}V_1\|_{L^\infty(\mathbb{R}^3)}
\\\nonumber
&&\lesssim \|\int_0^t e^{(t-\tau)\Delta } (-\Delta)^{-\frac{a}{2}}\mathbb{P}\nabla\cdot G_1 d\tau\|_{L^\infty(\mathbb{R}^3)}.
\end{eqnarray}

By using Proposition 3.5 and \eqref{3.47},  when  $0<\tau<\frac{t}{2}$  we have
\begin{eqnarray}\label{3.51}
&&
\|\int_0^{\frac{t}{2}} e^{(t-\tau)\Delta } (-\Delta)^{-\frac{a}{2}}\mathbb{P}\nabla\cdot G_1 d\tau d\tau\|_{L^\infty(\mathbb{R}^3)}
\\\nonumber
&&=
\|\int_0^{\frac{t}{2}}\mathbb{P}\nabla\cdot e^{\frac{(t-\tau)\Delta}2}\big(e^{\frac{(t-\tau)\Delta}2}(-\Delta)^{-\frac{a}{2}}G_1\big)d\tau\|_{L^\infty(\mathbb{R}^3)}\\\nonumber
&&\lesssim
\sup\limits_{y_0\in \mathbb{R}^3}\|\frac{t^{-\frac12}}{\sqrt{t}^3}\int_0^{\frac{t}{2}}\int_{B(y_0,\sqrt{t})}|e^{\frac{(t-\tau)\Delta}2}(-\Delta)^{-\frac{a}{2}}G_1|d\tau\|_{L^\infty(\mathbb{R}^3)}\\\nonumber
&&\lesssim
\sup\limits_{y_0\in \mathbb{R}^3}\|\frac{t^{-\frac12}}{\sqrt{t}^3}\int_0^{\frac{t}{2}}\int_{B(y_0,\sqrt{t})}\int_{\mathbb{R}^n}\frac{(t-\tau)^{\frac{a}{2}}}{\sqrt{t-\tau}^3}e^{-\frac{\tilde{y}^2}{2(t-\tau)}} |G_1(\tau,y-\tilde{y})|d\tilde{y}dyd\tau\|_{L^\infty(\mathbb{R}^3)}
\\\nonumber
&&\lesssim
\sup\limits_{y_0\in \mathbb{R}^3}\|\int_{\mathbb{R}^n}\frac{t^{\frac{a-1}{2}}}{\sqrt{t}^3}e^{-\frac{\tilde{y}^2}{2t}} \bigg(\frac1{\sqrt{t}^3}\int_0^{\frac{t}{2}}\int_{B(y_0+\tilde{y},\sqrt{t})} |G_1(\tau,y)|dyd\tau \bigg) d\tilde{y}\|_{L^\infty(\mathbb{R}^3)}\\\nonumber
&&\lesssim
  t^{\frac{a-1}{2}}\|G_1\|_{Y_T}.
\end{eqnarray}

When $t/2<\tau<t$, by using Lemma 2.1,  we have
\begin{eqnarray}\label{3.52}
&&
\|\int_{\frac{t}{2}}^t e^{(t-\tau)\Delta } (-\Delta)^{-\frac{a}{2}}\mathbb{P}\nabla\cdot G_1 d\tau d\tau\|_{L^\infty(\mathbb{R}^3)}
\\\nonumber
&&=
\|\int_{\frac{t}{2}}^t\mathbb{P}\nabla\cdot e^{\frac{(t-\tau)\Delta}2}\big(e^{\frac{(t-\tau)\Delta}2}(-\Delta)^{-\frac{a}{2}}G_1\big)d\tau\|_{L^\infty(\mathbb{R}^3)}\\\nonumber
&&\lesssim
 |\int_{\frac{t}{2}}^t\frac1{\sqrt{t-\tau}}\| e^{\frac{(t-\tau)\Delta}2}(-\Delta)^{-\frac{a}{2}}G_1 \|_{L^\infty(\mathbb{R}^3)}d\tau|\\\nonumber
 &&\lesssim
 |\int_{\frac{t}{2}}^t\frac1{\sqrt{t-\tau}}\| \int_{\mathbb{R}^3}\frac{(t-\tau)^{\frac{a}{2}}}{\sqrt{t-\tau}^3}e^{-\frac{(x-y)^2}{2(t-\tau)}} G_1(y)dy \|_{L^\infty(\mathbb{R}^3)}d\tau|\\\nonumber
 &&\lesssim
 |\int_{\frac{t}{2}}^t\frac{(t-\tau)^{\frac{a}{2}}}{\sqrt{t-\tau}}\| G_1(y)  \|_{L^\infty(\mathbb{R}^3)}d\tau| \lesssim \| G_1(y)  \|_{L^\infty(\mathbb{R}^3)} t^{\frac{a+1}{2}}\lesssim t^{\frac{a-1}{2}}\| G_1(y)  \|_{Y_T}.
\end{eqnarray}

Combining \eqref{3.51} -\eqref{3.52} and \eqref{3.48}, the Proposition 3.6 was proved.

\end{proof}

%\begin{corr}
%Let $G_2$ be periodic in $x$ and $\bbint_{\mathbb{T}^3}G_2(y,t) dy=0$, and  $V_2$ be the solution of the following system,
 %\begin{equation}\label{3.44}
%\begin{cases}
%V_{2t}-\triangle V_2 +\nabla P_2 =G_2\\
%\nabla\cdot V_2=0\\
%t=0: V_2=0.
%  \end{cases}
 %\end{equation}
 %and satisfying
 %\begin{equation}
 %\int_{T^3}V_2 (t,x)dx\equiv0,
 %\end{equation}
 %then  we have
 %\begin{equation} \label{3.45}
%\|V_2\|_{X_{T^* }}\lesssim  \| G_2\|_{Y_{T^* }}.
%\end{equation}
%\end{corr}

%\begin{proof}
%Since $\bbint_{\mathbb{T}^3}G_2(x+y,t) dy=0$, we have
% \begin{eqnarray} \label{3.46}
%&&G_2(x,t)=\bbint_{\mathbb{T}^3}(G_2(x,t)-G_2(x+y, t))dy \\
%\nonumber
%&&=- \bbint_{\mathbb{T}^3(0)} \int_0^1 \frac{d}{ds}G(t,x+sy) ds dy  =- \nabla_x \bigg(\int^1_0  \bbint_{\mathbb{T}^3} G(t,x+sz) \cdot zdzds\bigg).
%\end{eqnarray}
%Then we get our result by using Proposition 3.5.
%\end{proof}

\subsection{ The Time and Space Analyticity}

 We give the following corollaries beforehand.
\begin{corr} \label{cor3.2} For any integers $m, k\geq 0$, let $u$ be the solution of system \eqref{3.4}, then  we have
\begin{equation}\label{5.1}
\|t^{\frac{k+1}{2}+m}\partial_t^m\nabla^k u \|_{X_{T^*}}\lesssim\|u_0\|_{BMO^{-1}}.
\end{equation}
\end{corr}
\begin{proof}
Similar to the proof of Proposition \ref{proposition 3.1}, for  any $0\leq m\leq M$ and  $0\leq k\leq K$,
by Lemma 2.2,   we have:
\begin{eqnarray}\label{5.2}
&&t^{\frac{k+1}{2}+m}\|\partial_t^m\nabla^ku\|_{L^\infty(\mathbb{R}^3)}
= 4 t^{\frac{k-1}{2}+m}\|\int_{\frac{t}{8}}^{\frac{3t}8}\partial_t^m\nabla^kS(t-\tau)(e^{\tau\Delta}u_0)d\tau\|_{L^\infty(\mathbb{R}^3)}
\\\nonumber
&&\lesssim t^{\frac{k-1}{2}+m}\|\int_{\frac{t}{8}}^{\frac{3t}8}\big(\int_{\mathbb{R}^3}\frac{J^{k+2m}(\frac{y-\tilde{y}}{\sqrt{t-\tau}})}{\sqrt{t-\tau}^3}e^{-\frac{(y-\tilde{y})^2}{4(t-\tau)}}(e^{\tau\Delta}u_0)^2d\tilde{y}\big)^{1/2}   \frac1{\sqrt{t-\tau}^{k+2m}} d\tau\|_{L^\infty(\mathbb{R}^3)}
\\\nonumber
&&\lesssim\|\big(\sum\limits_{q=0}^\infty J^{k+2m}(q)e^{-\frac{q^2}{4}}\frac1{\sqrt{t}^3}\int_{\frac{t}{8}}^{\frac{3t}8}\int_{q\leq\frac{|y-\tilde{y}|}{\sqrt{t}}\leq q+1}
(e^{\tau\Delta}u_0)^2d\tilde{y}d\tau\big)^{1/2}\|_{L^\infty(\mathbb{R}^3)}\\\nonumber
&&\leq  C_0\|u_0\|_{BMO^{-1}}.
\end{eqnarray}
We still need to estimate the term $\sup\limits_{y_0\in \mathbb{R}^3,r>0}\big(r^{-n}\int_{Q(y_0,r)}(t^{\frac{k}{2}+m}\partial_t^m\nabla^ku)^2dydt\big)^{1/2}$. Here, if we estimate it directly, then we will have a non-integrable factor. By  using $\partial_t^m u=\partial_t^{m-1}\partial_t u=\Delta^mu$,  we get
\begin{eqnarray}\label{5.3}
&&\sup\limits_{y_0\in \mathbb{R}^3,r>0}\big(r^{-3}\int_{Q(y_0,r)}(t^{\frac{k}{2}+m}\partial_t^m\nabla^ku)^2dydt\big)^{1/2}
\\\nonumber
&&\lesssim   \sup\limits_{y_0\in \mathbb{R}^3,r>0} \big(r^{-n}\int_{Q(y_0,r)}(t^{\frac{k}{2}+m}\nabla^{k+2m}u)^2dydt\big)^{1/2}.
\end{eqnarray}
Therefore, it is sufficient to estimate the term $\sup\limits_{y_0\in \mathbb{R}^3,r>0}  \big(r^{-n}\int_{Q(y_0,r)}(t^{\frac{k}{2}}\nabla^ku)^2dydt\big)^{1/2}$ for any integer $k>0$.
\begin{eqnarray}\label{5.4}
&&\sup\limits_{y_0\in \mathbb{R}^3,r>0}\big(r^{-3}\int_{Q(y_0,r)} (t^{\frac{k}{2}}\nabla^ku)^2dydt\big)^{1/2}\\\nonumber
&&
\lesssim \sup\limits_{y_0\in \mathbb{R}^3,r>0}   r\| t^{\frac{k}{2}}\nabla^ku\|_{L^\infty(Q(y_0,r))}
 \\\nonumber
&&
\leq \sup\limits_{y_0\in \mathbb{R}^3,r>0}   \frac1r\| t^{\frac{k}{2}}\int_{\frac{r^2}{8}}^{\frac{3r^2}8} S(t-\tau) \nabla^k (e^{\tau\Delta}u_0)d\tau\|_{L^\infty(Q(y_0,r))}
\\\nonumber
&&\lesssim \sup\limits_{ y_0\in  \mathbb{R}^3, 0<t\leq r^2}     \|t^{\frac{k}{2}} \bigg(\int_{\frac{r^2}{8}}^{\frac{3r^2}8}  \|\nabla^k e^{\tau\Delta}u_0\|_{L^\infty(\mathbb{R}^3)}^2
d\tau\bigg)^{\frac12}
\|_{L^\infty(Q(y_0,r))}
\\\nonumber
&&\lesssim \sup\limits_{ y_0\in  \mathbb{R}^3, 0<t\leq r^2 }     t^{\frac{k}{2}} \bigg( \int_{\frac{r^2}{8}}^{\frac{3r^2}8} \frac{( \tau^{k/2}\|\nabla^{k} e^{\tau\Delta}u_0\|_{L^\infty(\mathbb{R}^3)})^2 }{\tau^{k/2}}
d\tau \bigg)^{\frac12}
\lesssim\|u_0\|_{BMO^{-1}}.
\end{eqnarray}
In the last inequality above, we have used \eqref{5.2}.

 \end{proof}

 \begin{corr}For any integers $m, k\geq 0$ any real constant $\alpha$, let $u$ be the solution of system \eqref{3.4}, then we have
\begin{equation}\label{5.1}
\|t^{\frac{k+\alpha}{2}+m}\partial_t^m\nabla^k u \|_{L^\infty }\lesssim\|u_0\|_{B^{-\alpha}_{\infty ,\infty}}.
\end{equation}
 \end{corr}
\begin{proof}
This result can be  verified by replacing $u$ as $\nabla^{-a}u$,  and then combing the process of Corollary 3.8 and the Theorem 2.2.3 in Danchin \cite{danchin}. See also P125 in Lemarie-Rieusset \cite{Lemarie-Rieusset}.
\end{proof}

Similarly to Proposition 3.4, we also have
\begin{corr}
 For any positive integers $m$ and $k$,  let $u$ be the solution of system \eqref{3.4}, then   we have
\begin{equation}\label{5.7}
\|t^{\frac{k}{2}+m}\partial_t^m\nabla^k(-\Delta)^{-1/4}u \|_{Z_{T^*}} \leq C_0 \|u_0\|_{BMO^{-1}}.
\end{equation}

\end{corr}

\begin{corr} \label{cor3.4}
Let $V_1$ and $G_1$  as given by Proposition 3.5,  for any integers $M, K\geq 0$ we have
\begin{equation}\label{5.5}
\|t^{\frac{k}{2}+m}\partial_t^m\nabla^k V_1 \|_{X_{T^*}}\lesssim   \|G_1\|_{Y_{T^*}},
\end{equation}
\end{corr}
\begin{proof}
This results can be verified similar to Corollary 5.1. See also in \cite{du1}.
\end{proof}

\begin{corr}
Let  Let $V_2$ and $G_2$  as given by Corollary 3.6,
  then for any positive integers $m$ and $k$, we have
 \begin{equation} \label{5.6}
\|t^{\frac{m}{2}+k}\partial_t^k\nabla^mV_2\|_{X_{T^* }}\lesssim  \|t^{\frac{m}{2}+k}\partial_t^k\nabla^m G_2\|_{Y_{T^* }}.
\end{equation}
\end{corr}

\section{Proof of Theorem 1.3 and Theorem 1.4}

We first consider the eigenvalue problem
\begin{equation}\label{4.1}
\begin{cases}
\nabla \times S=\mu S
\\
\nabla \cdot S=0.
\end{cases}
\end{equation}

Suppose that S is a periodic function on $\mathbb{T}^3$ such that
\begin{equation} \label{4.2}
\int_{\mathbb{T}^3} S(x) d x=0,
\end{equation}
then it is easy to see that in the space of divergence free vector field
\begin{displaymath}
(\nabla \times)^{-1}=(-\Delta)^{-1} \nabla \times.
\end{displaymath}
So $(\nabla \times)^{-1}$ is a compact operator, by the theory of eigenvalues of compact operators, we conclude that $(\ref{4.1})$ has eigenvalues $\{ \lambda_j \}_{j \in \mathbb{Z}}$ such that
\begin{equation} \label{4.3}
\cdots<\lambda_{-N}<\cdots<\lambda_{-1}<0<\lambda_{1}<\cdots<\lambda_{N}<\cdots
\end{equation}

Therefore, we can write
\begin{equation} \label{4.4}
u_0=\sum_{j=-\infty}^{+\infty} S_j
\end{equation}
where $S_j$ are the eigenfunctions with eigenvalues $\lambda_j$.

Take one more differentiation of $(\ref{4.1})$, we see that
\begin{equation} \label{4.5}
-\Delta S={\mu}^2 S.
\end{equation}
So $\mu^2$ is an eigenvalue of $-\Delta$ and S is an eigenvector of $-\Delta$. Thus, $(\ref{4.4})$ is nothing but a Fourier series expansion.

We are now ready to prove our theorem. Without loss of generality, we assume $\lambda\ge 0$, otherwise, a reflection $x\to -x$ reduce to this case.

Let
\begin{equation} \label{4.6}
u_{0+}=\sum_{j=1}^{+\infty} S_j, \  u_{0-}=\sum_{j=-\infty}^{-1} S_j
\end{equation}
then it is easy to see that
\begin{equation} \label{4.7}
u_{0+}=\frac{(u_0+(-\Delta)^{-\frac{1}{2}} \nabla \times {u_0})}{2},
\end{equation}
\begin{equation} \label{4.8}
u_{0-}=\frac{(u_0-(-\Delta)^{-\frac{1}{2}} \nabla \times {u_0})}{2},
\end{equation}
\begin{equation} \label{4.9}
(\sqrt{-\Delta}+\lambda) u_{0-}=\sum_{j=-\infty}^{-1} (-\lambda_j+\lambda) S_j = - \nabla \times u_{0-}+\lambda u_{0-}.
\end{equation}\

Therefore, there holds
\begin{eqnarray} \label{4.10}
&&{\|u_{0-}\|}_{{BMO}^{-1}} \lesssim {\|(\sqrt{-\Delta}+\lambda) u_{0-}\|}_{{BMO}^{-2}}\\ \nonumber
& & ={\|\nabla \times u_{0-}-\lambda u_{0-}\|}_{{BMO}^{-2}} \lesssim {\|\nabla \times u_{0}-\lambda u_{0}\|}_{{BMO}^{-2}} \lesssim \varepsilon {\langle \lambda \rangle}^{-b}.
\end{eqnarray}
In a similar way, we also have
\begin{eqnarray} \label{4.100}
&&{\|(\sqrt{-\Delta}-\lambda) u_{0+}\|}_{{BMO}^{-2}} \\ \nonumber
&&={\|\nabla \times u_{0+}-\lambda u_{0+}\|}_{{BMO}^{-2}}   \lesssim   {\|\nabla \times u_{0}-\lambda u_{0}\|}_{{BMO}^{-2}}  \lesssim   \varepsilon {\langle \lambda \rangle}^{-b}.
\end{eqnarray}

When $1\le \lambda_j\le \lambda/2$ or $\lambda_j\ge 2\lambda$, the operator $\sqrt{-\Delta}-\lambda$ is invertible.
Denote
\begin{equation}
u_{1+}=\sum_{1\le \lambda_j\le \lambda/2} S_j+\sum_{\lambda_j\ge 2\lambda}S_j,
\end{equation}
\begin{equation}
u_{2+}=\sum_{\lambda/2<\lambda_j< 2\lambda} S_j.
\end{equation}
Therefore we get in a similar way that
\begin{equation}\label{4.14}
\|u_{1+}\|_{BMO^{-1}} \lesssim \varepsilon {\langle \lambda \rangle}^{-b}.
\end{equation}

Denote
\begin{equation}
u_{01}=u_{0}-u_{2+}=u_{0-}+u_{1+},
\end{equation}
and combining \eqref{4.10} and \eqref{4.14}, we have
\begin{equation}\label{4.16}
\|u_{01}\|_{BMO^{-1}}\lesssim  \varepsilon {\langle \lambda \rangle}^{-b}.
\end{equation}

Let
\begin{equation}\label{4.17}
  U=u+v,
\end{equation}
with
\begin{equation}\label{4.18}
u=e^{t\Delta} u_{2+}.
\end{equation}
Then we rewrite the system \eqref{1.1} as the following system with small data:
\begin{equation}\label{4.19}
\begin{cases}
v_t+(v \cdot \nabla)v+(u \cdot \nabla)v+(v \cdot \nabla)u+(u \cdot \nabla)u+\nabla P=\Delta v,
\\
t=0:\ v=u_{01}.
\end{cases}.
\end{equation}

\subsection{Step 1: Short time existence}
  The purpose of the short time existence result is just to break the scaling. We shall study our problem by the classical fixed point argument. \\
\indent First, we introduce the following space:
\begin{deff}
Let   $T_1>0$ be  a given constant and $\epsilon_0$  be a small constant, we say $f\in \mathbb{E}_{T_1, \epsilon_0}$, if the following holds:
\begin{itemize}
  \item[(i):] $f(t,x)$ is a periodic function of $x$ on $\mathbb{T}^3$; %and satisfying $\int_{\mathbb{T}^3}f(x)dx=0$;
   \item[(ii):] $\nabla \cdot f=0$,  $\|f\|_{ \mathbb{X}_{T_1} } \leq
 \epsilon_0<\lambda>^{-b}$.
\end{itemize}
%For simplicity, we write $ \|f\|_{ \mathbb{X}_1 }$ as $ \|f\|_{ \mathbb{X} }$ and   $\mathbb{E}_{1,\epsilon}$ as   $\mathbb{E}_{\epsilon}.$
\end{deff}
We define
$v=  \mathfrak{F}\tilde{v}$ by solving the following linear equation with $\tilde{v}\in \mathbb{E}_{T_1,\epsilon_0} $
\begin{equation}
\begin{cases}
v_t -\Delta v +\nabla P=-((u+\tilde{v}) \cdot \nabla)(u+\tilde{v}),\\
t=0:\ v=u_{01}.
\end{cases}
\end{equation}

Let $
g=u+\lambda {\Delta}^{-1}(\nabla \times u),
$
then
\begin{equation} \label{4.21}
\nabla \times u- \lambda u= \nabla \times g
\end{equation}
and
\begin{equation} \label{4.22}
u=g-\lambda {\Delta}^{-1} (\nabla \times u).
\end{equation}

By Proposition 3.2 and Proposition 3.5, we get
\begin{eqnarray} \label{4.23}
&& \|v\|_{X_{T_1} } \lesssim \|(u+\tilde{v})\otimes(u+\tilde{v}) \|_{Y_{T_1} } + {\|u_{0-}\|}_{{BMO}^{-1}}
 \\\nonumber
&& \lesssim \left( {\| u \|}_{X_{T_1}}+{\| \tilde{v} \|}_{X_{T_1}} \right)^2 + \varepsilon {\langle \lambda \rangle}^{-b}
\lesssim \left( {\| u \|}_{X_{T_1}}+\varepsilon_0 {\langle \lambda \rangle}^{-b} \right)^2 + \varepsilon {\langle \lambda \rangle}^{-b}
.
\end{eqnarray}

By Proposition 3.3 and Corollary 3.4, we have
\begin{eqnarray} \label{4.24}
&& {\|u \|}_{X_{T_1}} \lesssim {\|g \|}_{X_{T_1}}+\lambda {\| {\Delta}^{-1} \nabla \times u \|}_{X_{T_1}} \\ \nonumber
&& \lesssim {\| \nabla \times u_0-\lambda u_0 \|}_{{BMO}^{-2}}+\lambda {T_1}^{\frac{1}{4}} {\| \sqrt{-\Delta} {\Delta}^{-1} \nabla \times u_0 \|}_{{BMO}^{-1}} \\ \nonumber
& &\lesssim \varepsilon {\langle \lambda \rangle}^{-b}+\lambda {T_1}^{\frac{1}{4}} M_0.
\end{eqnarray}

We take $\varepsilon_0=C(M_0)^{-1}\varepsilon$ and take $T_1$   as
\begin{equation}\label{4.25}
T_1= C(M_0)\varepsilon^2  {\langle \lambda \rangle}^{-2b-4}
\end{equation}
where the constant  $C(M_0)$ is independent of $\lambda$ and $\varepsilon$ and $C(M_0)$ is choosing suitably small.

Then by using \eqref{4.23}-\eqref{4.25} we get
\begin{equation} \label{4.26}
\|v\|_{X_{T_1} } \leq \varepsilon_0 {\langle \lambda \rangle}^{-b}.
\end{equation}

 \indent By a similar process, we take $v_1$ be the corresponding solution to $(\ref{4.19})$ where $\tilde{v_1} \in E_{T_1,\varepsilon_0}$, and $v_2$ be the solution with $\tilde{v_2} \in E_{T_1,\varepsilon_0}$.
 Denote
 \begin{equation}
 \bar{\tilde{v}}=\tilde{v_1}-\tilde{v_2},\ \bar{v}=v_1-v_2,\ \bar{q}=P_1-P_2,
 \end{equation}
 then there hold
 \begin{equation}
 \bar{v}_t-\Delta \bar{v}+\nabla \bar{q}=(\tilde{v_2} \cdot \nabla) \bar{\tilde{v}}+(\bar{\tilde{v}} \cdot \nabla) \tilde{v_1}+(u \cdot \nabla \bar{\tilde{v}})+(\bar{\tilde{v}} \cdot \nabla)u.
 \end{equation}

 Therefore, we get
 \begin{equation} \label{4.20}
 {\| \bar{v} \|}_{X_{T_1}} \lesssim  {\| \bar{\tilde{v}} \|}_{X_{T_1}}  \left(  {\| \tilde{v_2} \|}_{X_{T_1}} + {\| \tilde{v_1} \|}_{X_{T_1}} + {\| u \|}_{X_{T_1}} \right).
 \end{equation}
 Since $\epsilon_0$ is small enough, then
 \begin{equation} \label{4.21}
  {\| \tilde{v_2} \|}_{X_{T_1}}+   {\| \tilde{v_1} \|}_{X_{T_1}}+  {\| u \|}_{X_{T_1}} <<1.
 \end{equation}
Combining \eqref{4.20} and \eqref{4.21},  we conclude that the mapping is a contraction.

  Our next goal is to prove the space and time analyticity of the solution.

 \begin{deff}
 Let $g$ be a function defined on $\mathbb{R}^n\times [0,T_1)$,  for any integers $M, K\geq0$, we say $g\in \mathbb{X}_{T_1}^{M,K}$, if
\begin{multline}\label{4.22}
\|g\|_{ \mathbb{X}_{T_1}^{M,K}}\triangleq\sum\limits_{m=0}^M\sum\limits_{k=0}^K\bigg(\sup\limits_{t}t^{\frac{k+1}{2}+m}\|\partial_t^m\nabla^{k}g\|_{L^\infty(\mathbb{R}^3)}
\\
+\sup\limits_{y_0\in \mathbb{R}^3,r>0}\bigg(r^{-3}\int_{Q(y_0,r)}|t^{\frac{k}{2}+m}\partial_t^m\nabla^{k} g|^2dy  dt
\bigg)^{\frac12}\bigg)<+\infty.
\end{multline}
\end{deff}

Subsequently, we give  the following function spaces
\begin{deff}
Let   $T>0$ be  a given constant and $\epsilon_0$ as in Theorem 1.3,   for any integers $M,K>0$ we say $f\in \mathbb{E}^{M,K}_{T, \epsilon_0}$, if the following holds:
\begin{itemize}
  \item[(i):] $f(x,t)$ is a space  periodic function on $\mathbb{T}^3$; %and satisfying $\int_{\mathbb{T}^3}f(x)dx=0$;
  \item[(ii):] $\nabla\cdot f=0$, and $\|f(x,t)\|_{ \mathbb{X}_{T}^{M,K}} \leq
 \epsilon_0<\lambda >^{-b}$.
  %  \item[(iii):] $\nabla \cdot f=0\quad \mbox{and} \quad \|(\mathbb{I}+\lambda \Delta^{-1}\nabla \times)f\|_{ \mathbb{X^{M,K}_T} }  <
% \epsilon_0$.
\end{itemize}
%For simplicity, we write $ \|f\|_{ \mathbb{X}_1 }$ as $ \|f\|_{ \mathbb{X} }$ and   $\mathbb{E}_{1,\epsilon}$ as   $\mathbb{E}_{\epsilon}.$
\end{deff}

Repeat the above process, we can show that $v \in \mathbb{E}^{M,K}_{T_1,\varepsilon_0}$.

\subsection{Step 2:  Local in time existence}

Now since $v$ is analytic at time $T_1$, so it can be extended to some time interval $[T_1,T_2]$. We shall give an estimate of $T_2$. For that purpose, we only need to give an apriori estimate for the ${\|v(t)\|}_{L^{\infty}}$ on the time interval $[T_1,T_2]$.

By Proposition 3.3 and Corollary 3.6, we get, for $0<a<1$,
\begin{equation}
t^{\frac{1-a}{2}} {\| {(-\Delta)}^{-\frac{a}{2}} v(t,\cdot) \|}_{L^{\infty}} \lesssim \varepsilon {\langle \lambda \rangle}^{-b}.
\end{equation}
Moreover, take $t=T_1$, we get
\begin{eqnarray*}
{\| {(-\Delta)}^{-\frac{a}{2}} v(T_1,\cdot) \|}_{L^{\infty}} & \lesssim & \varepsilon {\langle \lambda \rangle}^{-b} { T_1}^{-\frac{1-a}{2}} \\
& \lesssim & \varepsilon {\langle \lambda \rangle}^{-b} {\langle \lambda \rangle}^{(b+2)(1-a)} \varepsilon^{-(1-a)}.
\end{eqnarray*}

Taking $1-a=\frac{b}{b+2}$, we get
\begin{equation}\label{4.100}
{\| {(-\Delta)}^{-\frac{a}{2}} v(T_1,\cdot) \|}_{L^{\infty}} \leq C_0(M_0) \varepsilon^a
\end{equation}

By Lemma 2.4, we rewrite $(\ref{4.19})$ as
\begin{equation} \label{4.230}
\begin{cases}
v_t+(v \cdot \nabla)v+(u \cdot \nabla)v+(v \cdot \nabla)u+\nabla \left( P+\frac{{|u|}^2}{2} \right)-\Delta v=u \times(\nabla \times u),
\\
t=T_1:\ v=v(T_1,x).
\end{cases}
\end{equation}
By the divergence free condition and Duharmel's formula, we get
\begin{eqnarray}\label{101}
v(t)=e^{\Delta(t-T_1)}v(T_1)+\int_{T_1}^t\tilde{\mathbb{P}}(t-\tau ,\cdot)\ast (u\times(\nabla\times u))(\tau )d\tau \\\nonumber
+\int_{T_1}^t\nabla\tilde{\mathbb{P}}(t-\tau )\ast
(v\otimes v+v\otimes u+u\otimes v)(\tau)d\tau\\\nonumber
=I +II+III.
\end{eqnarray}
Where $\tilde{\mathbb{P}}$ is defined in Lemma 2.1.

Noting \ref{4.100}, it follows from Corollary 3.9 (in which we take $m=k=0$, $\alpha =a$)that
\begin{equation}
\|I\|_{L^\infty}\lesssim {(t-T_1)}^{-\frac{a}{2}} \|v(T_1,\cdot)\|_{B^{-a}_{\infty ,\infty}}\lesssim {(t-T_1)}^{-\frac{a}{2}} \varepsilon^a
\end{equation}
Besides, we have
\begin{eqnarray*}
&&u \times(\nabla \times u)= u \times(\nabla \times u-\lambda u) \\
&&= \sum_{\lambda/2<\lambda_j<2\lambda}e^{-\lambda_j^2t} S_j \times \sum_{\lambda/2<\lambda_k<2\lambda} e^{-\lambda_k^2t}(\lambda_k-\lambda)S_k \\
&&= \sum_ {\lambda/2<\lambda_j ,\lambda_k<2\lambda}e^{-(\lambda_j^2+\lambda_k^2)t}(\lambda_k-\lambda) S_j \times S_k \\
&&= \sum_{\lambda/2<\lambda_j ,\lambda_k<2\lambda}e^{-(\lambda_j^2+\lambda_k^2)t} \frac{\lambda_k-\lambda}{\lambda_j+\lambda_k} \left[ \left( \nabla \times S_j \right)\times S_k+S_j \times \left( \nabla \times S_k \right) \right] \\
&&= \nabla \left( \sum_{\lambda/2<\lambda_j ,\lambda_k<2\lambda}
 e^{-(\lambda_j^2+\lambda_k^2)t}\frac{\lambda_k-\lambda}{\lambda_j+\lambda_k} S_j \cdot S_k \right)\\
&&+\nabla \cdot \left[ \sum_{\lambda/2<\lambda_j ,\lambda_k<2\lambda} \frac{e^{-(\lambda_j^2+\lambda_k^2)t}}{\lambda_j+\lambda_k} \left( S_j \otimes \left( \nabla \times S_k- \lambda S_k \right) + \left( \nabla \times S_k - \lambda S_k \right) \otimes S_j \right) \right].
\end{eqnarray*}
Then it follows from  Corollary 3.9  (in which we take $k=m=0$ $\alpha=1-\kappa$ )
that
\begin{equation}
\|II\|_{L^\infty}\lesssim \int_{T_1}^t {(t-\tau)}^{-\frac{1-\kappa}{2}} {\|F(\tau)\|}_{B^{\kappa}_{\infty,\infty}} d \tau
\end{equation}
here, we take $0<\kappa <1-a$ and
\begin{displaymath}
F(\tau)=\sum_{\lambda/2<\lambda_j ,\lambda_k<2\lambda} e^{-(\lambda_j^2+\lambda_k^2)t}\frac{1}{\lambda_j+\lambda_k} \left( S_j \otimes \left( \nabla \times S_k- \lambda S_k \right) + \left( \nabla \times S_k - \lambda S_k \right) \otimes S_j \right).
\end{displaymath}
On the other hand, we have
\begin{eqnarray}
\|III\|_{L^\infty}&\le& \int_{T_1}^t\|\nabla\tilde{\mathbb{P}}(t-\tau,\cdot )\|_{L^1}
\|(v\otimes v+v\otimes u+u\otimes v)(\tau)\|_{L^\infty}d\tau\\\nonumber
&\lesssim& \int_{T_1}^t(t-\tau )^{-\frac{1}{2}}
\|(v\otimes v+v\otimes u+u\otimes v)(\tau)\|_{L^\infty}d\tau.
\end{eqnarray}

Summarizing, we get
\begin{eqnarray}
&& {\|v(t,\cdot)\|}_{L^{\infty}} \leq C_0{(t-T_1)}^{-\frac{a}{2}} \varepsilon^a+C_0 \int_{T_1}^t {(t-\tau)}^{-\frac{1}{2}}  {\|v(\tau,\cdot)\|}_{L^{\infty}}^2d\tau
 \\\nonumber
&&
+C_0\int_{T_1}^t (t-\tau)^{-1/2}\|v(\tau,\cdot )\|_{L^\infty} \|u(\tau,\cdot)\|_{L^\infty} d\tau +C_0\int_{T_1}^t {(t-\tau)}^{-\frac{1-\kappa}{2}} {\|F(\tau)\|}_{B^{\kappa}_{\infty,\infty}} d \tau.
\end{eqnarray}

By Corollary 2.6, we get
\begin{eqnarray*}
{\|F(\tau,\cdot)\|}_{B^{\kappa}_{\infty,\infty}} & \leq & {\|\nabla \times u-\lambda u\|}_{B^{\kappa -1+a}_{\infty,\infty}}{\|u\|}_{B^{-a}_{\infty,\infty}} \\
& \leq & \tau^{-\frac{1+a+\kappa}{2}} {\|\nabla \times u_0-\lambda u_0\|}_{BMO^{-2}} {\tau}^{-\frac{1-a}{2}} {\|u_0\|}_{BMO^{-1}} \\
& \leq & {(\tau-T_1)}^{-\frac{1+a+\kappa}{2}} T_1^{-\frac{1-a}{2}}{\|\nabla \times u_0-\lambda u_0\|}_{{BMO}^{-2}} {\|u_0\|}_{{BMO}^{-1}} \\
& \lesssim & {(\tau-T_1)}^{-\frac{1+a+\kappa }{2}} C(M_0)
\varepsilon^a.
\end{eqnarray*}
Therefore, there holds
\begin{eqnarray*}
 && \int_{T_1}^t {(t-\tau)}^{-\frac{1-\kappa}{2}} {\|F(\tau)\|}_{B^{\kappa}_{\infty,\infty}} d \tau
 \\\nonumber
 & &\lesssim  \varepsilon^a \int_{T_1}^t {(t-\tau)}^{-\frac{1-\kappa}{2}}{(\tau-T_1)}^{-\frac{1+a+\kappa }{2}}d\tau\\\nonumber
 & &\lesssim \varepsilon^a{(t-T_1)}^{-\frac{1-\kappa}{2}}\int_{T_1}^{T_1+\frac{t-T_1}{2}} {(\tau-T_1)}^{-\frac{1+a+\kappa }{2}}d\tau
 \\\nonumber
 & &
 +\varepsilon^a
{(t-T_1)}^{-\frac{1+a+\kappa }{2}}\int_{T_1+\frac{t-T_1}{2}}^t {(t-\tau)}^{-\frac{1-\kappa}{2}}d\tau \lesssim  \varepsilon^a{(t-T_1)}^{-\frac{a}{2}}.
\end{eqnarray*}

We also have
\begin{eqnarray*}
&&\int_{T_1}^t (t-\tau)^{-1/2}\|v(\tau,\cdot)\|_{L^\infty} \|u(\tau,\cdot)\|_{L^\infty} d\tau  \\
&& \lesssim C_0\int_{T_1+(1-\delta)(t-T_1)}^t(t-\tau)^{-\frac{1}{2}}
\|u(\tau,\cdot)\|_{L^\infty}\|v(\tau,\cdot)\|_{L^\infty}d\tau\\\
&&+C(\delta)(t-T_1)^{-\frac{1}{2}}\int_{T_1}^{T_1+(1-\delta)(t-T_1)}
\|u(\tau,\cdot)\|_{L^\infty}\|v(\tau,\cdot)\|_{L^\infty}d\tau\\\
&&\lesssim C_1M_0(t-T_1)^{-\frac{a}{2}}\sup_t(t-T_1)^{\frac{a}{2}}\|v(\tau,\cdot)\|_{L^\infty}
\int_{T_1+(1-\delta)(t-T_1)}^t(t-\tau)^{-\frac{1}{2}}\tau^{-\frac{1}{2}}d\tau\\\
&&+C(\delta)(t-T_1)^{-\frac{a}{2}}\left(\int_{T_1}^{T_1+(1-\delta)(t-T_1)} \|u(\tau,\cdot)\|_{L^\infty}^2((\tau-T_1)^{\frac{a}{2}}\|v(\tau,\cdot)\|_{L^\infty})^2d\tau\right)^{1/2}.
\end{eqnarray*}
By let $(\tau-T_1)=\theta(t-T_1)$, we get
\begin{equation}
\int_{T_1+(1-\delta )(t-T_1)}^t(t-\tau)^{-\frac{1}{2}}\tau^{-\frac{1}{2}}d\tau\le
\int_{1-\delta }^1 (1-\theta)^{-\frac{1}{2}}\theta^{-\frac{1}{2}}ds
\end{equation}
which can be small if $\delta$ is made sufficiently small. We take
\begin{equation}
2C_1M_0\int_{T_1+(1-\delta)(t-T_1)}^t(t-\tau)^{-\frac{1}{2}}\tau^{-\frac{1}{2}}d\tau\le 1/2.
\end{equation}

On the other hand, it is easy to see
\begin{eqnarray}\int_{T_1}^t {(t-\tau)}^{-\frac{1}{2}}  {\|v(\tau,\cdot)\|}_{L^{\infty}}^2d\tau\le\left(\int_{T_1}^t {(t-\tau)}^{-\frac{1}{2}} \tau^{-a} d\tau\right) \sup_{T_1\le \tau\le t}
\tau^a|v(\tau)|_{L^{\infty}}^2.
\end{eqnarray}

Moreover, there also holds
\begin{eqnarray}
\int_{T_1}^t {(t-\tau)}^{-\frac{1}{2}} \tau^{-a} d\tau&\lesssim&{(t-T_1)}^{-\frac{1}{2}}\int_{T_1}^{T_1+\frac{t-T_1}{2}}  \tau^{-a} d\tau\\\nonumber
&&+{(t-T_1)}^{-a}\int_{T_1+\frac{t-T_1}{2}}^t {(t-\tau)}^{-\frac{1}{2}}  d\tau\lesssim {(t-T_1)}^{\frac{1}{2}-a}.
\end{eqnarray}

Thus, we get
\begin{eqnarray*}
\sup_{T_1\le t_1\le t}{(t_1-T_1)}^{\frac{a}{2}} {\|v(t_1,\cdot)\|}_{L^{\infty}}& \lesssim &   \varepsilon^a
+{(t-T_1)}^{\frac{1-a}{2}} \left( \sup\limits_\tau {\left( {(\tau-T_1)}^{\frac{a}{2}} {\|v(\tau,\cdot)\|}_{L^{\infty}} \right)}^2 \right.   \\
&&+   \left. \left(\int_{T_1}^{t} \|u(\tau,\cdot)\|_{L^\infty }^2\big((\tau-T_1)^{\frac{a}{2}}\|v(\tau,\cdot)\|_{L^\infty}\big)^2d\tau\right)^{\frac{1}{2}} \right).
\end{eqnarray*}

Take $T_2=T_1+1$, for $T_1\le t\le T_2$, we get
\begin{eqnarray}\label{4.44}
&&\sup_{T_1\le t_1\le t}{(t_1-T_1)}^{\frac{a}{2}} {\|v(t_1,\cdot)\|}_{L^{\infty}} \\\nonumber
&&\leq 2C_2 \left[ \varepsilon^a
+   \left(\int_{T_1}^{t} \|u(\tau,\cdot)\|_{L^\infty }^2((\tau-T_1)^{\frac{a}{2}}\|v(\tau,\cdot)\|_{L^\infty})^2d\tau\right)^{\frac{1}{2}}
 \right]
\end{eqnarray}
provided that
\begin{equation}\label{4.50}
C_2\sup_{T_1\le t_1\le t}{(t_1-T_1)}^{\frac{a}{2}} {\|v(t_1,\cdot)\|}_{L^{\infty}}\le \frac{1}{2}.
\end{equation}

Furthermore, we have
\begin{equation}
\left(\int_0^{t} \|u(\tau,\cdot)\|_{L^\infty}^2d\tau\right)^{\frac{1}{2}}\lesssim \|u_{2+}\|_{B^{-1}_{\infty ,2}}.
\end{equation}
However, $u_{2+}$ has only finite piece of the littlewood-Paley decomposition, thus
\begin{equation}
\|u_{2+}\|_{B^{-1}_{\infty ,2}}\lesssim \|u_{2+}\|_{B^{-1}_{\infty ,\infty}}\lesssim \|u_0\|_{BMO^{-1}}.
\end{equation}

Then from \eqref{4.44} and the Gronwall's inequality, we have
\begin{displaymath}
\sup\limits_{T_1 \leq t \leq T_2} {(t-T_1)}^{\frac{a}{2}} {\|v(t,\cdot)\|}_{L^{\infty}} \leq 4C_2 \varepsilon^a,
\end{displaymath}
provided that $\varepsilon$ is sufficiently small. This shows that the assumption \eqref{4.50} is reasonable. Therefore, we can extend our solution to time $T_2$.

By a similar argument, we can get
\begin{displaymath}
\sup\limits_{T_1 \leq t \leq T_2} {(t-T_1)}^{\frac{a}{2}+m+\frac{k}{2}}  {|\partial_t^m \nabla^k v(t,\cdot)|}_{L^{\infty}} \leq C_{m,k} \varepsilon^a
\end{displaymath}
for any $m \geq 0,\ k \geq 0$, which implies
\begin{displaymath}
 {\| \nabla^k v(T_2,\cdot)\|}_{L^{\infty}} \lesssim C(k,M_0) \varepsilon^a.
\end{displaymath}
\subsection{\bf Step 3: Global existence}
We now consider the solution on a fixed periodic $x \in \left[-\pi,\pi \right]^3$. To prove global existence on the time interval $[ T_2,+\infty )$, we only need to give an apriori $H^1$ bound of the solution $v$. We multiply the equation $(\ref{4.230})$ by $\Delta v$ and integration by parts to get
\begin{eqnarray*}
&&\frac{1}{2} {\| \nabla v(t,\cdot) \|}_{L^2}^2+\int_{T_2}^t  {\| \Delta v(\tau,\cdot) \|}_{L^2}^2 d \tau\\\nonumber
 & &\leq \frac{1}{2} {\| \nabla v(T_2,\cdot) \|}_{L^2}^2+\int_{T_2}^t  {\| \Delta v(\tau,\cdot) \|}_{L^2}^2  +{\| v \cdot \nabla v +u \cdot \nabla v+v \cdot \nabla u\|}_{L^2}^2 (\tau) d \tau,
\end{eqnarray*}
and
\begin{eqnarray*}
&&{\| v \cdot \nabla v +u \cdot \nabla v+v \cdot \nabla u\|}_{L^2} \\
& & \lesssim\left( {\|v\|}_{L^{\infty}}+{\|u\|}_{L^{\infty}} \right) {\|\nabla v\|}_{L^{2}}+{\|v\|}_{L^{\infty}}{\|\nabla u\|}_{L^{2}} \\
& &\lesssim \left( {\|v\|}_{H^{2}}^{\frac{1}{2}}{\|v\|}_{H^{1}} ^{\frac{1}{2}}+{\|u\|}_{L^{\infty}} \right) {\|\nabla v\|}_{L^{2}}+{\|\nabla u\|}_{L^{2}} {\|v\|}_{H^{2}}^{\frac{1}{2}}{\|v\|}_{H^{1}} ^{\frac{1}{2}}.
\end{eqnarray*}
By Poincare's inequality, we get
\begin{eqnarray*}
&&{\| v \cdot \nabla v +u \cdot \nabla v+v \cdot \nabla u\|}_{L^2}\\\nonumber
 & & \lesssim {\|\nabla v\|}_{L^{2}}^{\frac{3}{2}}  {\|\Delta v\|}_{L^{2}}^{\frac{1}{2}}+ {\|\nabla v\|}_{L^{2}}^{\frac{1}{2}} {\|\nabla u\|}_{L^{2}}  {\|\Delta v\|}_{L^{2}}^{\frac{1}{2}}+ {\|\nabla v\|}_{L^{2}}{\|u\|}_{L^{\infty}} \\
& &\lesssim {\|\nabla v\|}_{L^{2}}  {\|\Delta v\|}_{L^{2}}+{\|\nabla v\|}_{L^{2}}^{\frac{1}{2}} {\|\nabla u\|}_{L^{2}}  {\|\Delta v\|}_{L^{2}}^{\frac{1}{2}}+ {\|\nabla v\|}_{L^{2}}{\|u\|}_{L^{\infty}}.
\end{eqnarray*}

Therefore, we get
\begin{eqnarray*}
&&\frac{1}{2} {\|\nabla v (t,\cdot)\|}_{L^2}^2 + \int_{T_2}^t {\|\Delta v (\tau,\cdot)\|}_{L^2}^2 d \tau\\
 && \lesssim  \varepsilon^{2a}+\frac{1}{2}\int_{T_2}^t {\|\Delta v (\tau,\cdot)\|}_{L^2}^2 d \tau+\int_{T_2}^t {\|\Delta v (\tau,\cdot)\|}_{L^2}^2 {|\nabla v |}_{L^2} d \tau \\
&& +\int_{T_2}^t  {\|\nabla v \|}_{L^2}^2  {\|\nabla u \|}_{L^2}^4+ {\|\nabla v \|}_{L^2}^2 {\|u\|}_{L^{\infty}}^2 d \tau.
\end{eqnarray*}
We have
\begin{eqnarray*}
&&\int_{T_2}^t {\|\nabla u \|}_{L^2({\mathbb{T}}^3)}^4+ {\|u\|}_{L^{\infty}({\mathbb{T}}^3)}^2 d \tau\\
 && \lesssim \int_{T_2}^t {\|\nabla u \|}_{L^{\infty}({\mathbb{T}}^3)}^4+ {\|u\|}_{L^{\infty}({\mathbb{T}}^3)}^2 d \tau \\
&& \lesssim \int_{T_2}^t \left( {\tau}^{-4} {\|u_0\|}_{{BMO}^{-1}}+{\tau}^{-2} {\|u_0\|}_{{BMO}^{-2}} \right) d \tau \\
&& \leq  C M_0 \left( {T_2}^{-3}+{T_2}^{-1} \right)  \leq  C_3(M_0).
\end{eqnarray*}

Therefore, by a bootstrap argument, it follows from Gronwall's inequality that
\begin{displaymath}
{\|\nabla v (t,\cdot)\|}_{L^2}^2+\int_{T_2}^t {\|\Delta v (\tau,\cdot)\|}_{L^2}^2 d \tau \leq C_3(M_0) \varepsilon^{2a}.
\end{displaymath}
This shows that the solution is global.
\section*{Acknowledgement}  Yi Du was supported
by NSFC (grant No. 11471126). Yi Zhou was supported by Key Laboratory of Mathematics for Nonlinear Sciences (Fudan University), Ministry of Education of China, P.R.China.
Shanghai Key Laboratory for Contemporary Applied Mathematics, School of Mathematical Sciences, Fudan University, P.R. China, NSFC (grants No. 11421061), 973 program (grant No. 2013CB834100) and 111 project.

%%%%%%%%%%%%%%%%%%%%%%%%%%%%%%%%%%%%%%%%%%%%%%%%%%%%%%%%%%%%%%%%%%%%%%%%%%%%%%%%%%%%%%%%%%%%%%%%%%%%%%%%%%%%%%%%%%%%%%%%%%%%%%%%%%%%%%%%%5

\end{document}